\documentclass[leqno,onefignum,onetabnum]{siamltex1213}
\usepackage{amsmath}
\newcommand{\mb}[1]{\ensuremath{\mathbf{#1}}}
\newcommand{\fnc}[1]{\ensuremath{\mathcal{#1}}}
\newcommand{\mr}[1]{\ensuremath{\mathrm{#1}}}
\newcommand{\mat}[1]{\ensuremath{\mathsf{#1}}}

\usepackage{amssymb}
\usepackage{tikz}
\usetikzlibrary{matrix,backgrounds}
\usepackage{soul}
 
 \usepackage{subfigure}
\usepackage{pdflscape}
 \usepackage{verbatim}
 \usepackage{listings}
 \usepackage{color}

\definecolor{dkgreen}{rgb}{0,0.6,0}
\definecolor{gray}{rgb}{0.5,0.5,0.5}
\definecolor{mauve}{rgb}{0.58,0,0.82}
\title{Generalized Summation-by-Parts Operators for the Second Derivative with Variable Coefficients} 
\author{
David C. Del Rey Fern\'andez
\footnote[2]{\lowercase{\uppercase{P}h.D. \uppercase{C}andidate, \uppercase{I}nstitute for \uppercase{A}erospace \uppercase{S}tudies, \uppercase{U}niversity of \uppercase{T}oronto, \uppercase{T}oronto, \uppercase{O}ntario, \uppercase{M}3\uppercase{H} 5\uppercase{T}6, \uppercase{C}anada (\email{dcdelrey@gmail.com})}.} 
and David W. Zingg
\footnote[3]
{
\lowercase{
\uppercase{P}rofessor and \uppercase{D}irector, \uppercase{T}ier 1 \uppercase{C}anada \uppercase{R}esearch \uppercase{C}hair in \uppercase{C}omputational \uppercase{A}erodynamics and \uppercase{E}nvironmentally-\uppercase{F}riendly \uppercase{A}ircraft \uppercase{D}esign, \uppercase{J}. \uppercase{A}rmand \uppercase{B}ombardier \uppercase{F}oundation \uppercase{C}hair in \uppercase{A}erospace \uppercase{F}light \uppercase{I}nstitute for \uppercase{A}erospace \uppercase{S}tudies, \uppercase{U}niversity of \uppercase{T}oronto, \uppercase{T}oronto, \uppercase{O}ntario, \uppercase{M}3\uppercase{H} 5\uppercase{T}6, \uppercase{C}anada (\email{dwz@oddjob.utias.utoronto.ca}).
}
}
}

\begin{document}
\maketitle
\slugger{sisc}{xxxx}{xx}{x}{x--x}

\begin{abstract}
The comprehensive generalization of summation-by-parts of Del Rey Fern\'andez et al.\ (J. Comput. Phys., 266, 2014) is extended to approximations of second derivatives with variable coefficients. This enables the construction of second-derivative operators with one or more of the following characteristics: i) non-repeating interior stencil, ii) nonuniform nodal distributions, and iii) exclusion of one or both boundary nodes. Definitions are proposed that give rise to generalized SBP operators that result in consistent, conservative, and stable discretizations of PDEs with or without mixed derivatives. It is proven that such operators can be constructed using a correction to the application of the first-derivative operator twice that is the same as used for the constant-coefficient operator. Moreover, for operators with a  repeating interior stencil, a decomposition is proposed that makes the application of such operators particularly simple. A number of novel operators are constructed, including operators on pseudo-spectral nodal distributions and operators that have a repeating interior stencil, but unequal nodal spacing near boundaries. The various operators are  compared to the application of the first-derivative operator twice in the context of the linear convection-diffusion equation with constant and variable coefficients.
\end{abstract}



\pagestyle{myheadings}
\thispagestyle{plain}
\markboth{Generalized Summation-By-Parts Operators for Second Derivatives}{David C. Del Rey Fern\'andez and David W. Zingg}
\section{Introduction}\label{Introduction}
The focus of this paper is on developing consistent, conservative, and provably stable high-order approximations of the second derivative with variable coefficients. One such methodology is the combination of summation-by-parts (SBP) operators with boundary conditions and inter-element coupling weakly enforced using simultaneous approximation terms (SATs) \cite{Carpenter1994,Carpenter1999,Nordstrom1999,Nordstrom2001b,Mattsson2004b,Mattsson2012}. 

The most straightforward means of approximating the second derivative is to apply a first-derivative operator twice. However, this has the drawback that the resultant operator is one order less accurate than the first derivative \cite{Mattsson2004b,Mattsson2008}. Alternatively, operators of the same order of accuracy as the first derivative can be constructed. For the classical finite-difference (FD) SBP method, it is possible to construct minimum-stencil operators which have the same bandwidth and order as first-derivative operators. Classical minimum-stencil FD-SBP operators were first investigated by Mattsson and Nordstr\"om \cite{Mattsson2004b} and refined by Mattsson, Sv\"ard, and Shoeybi \cite{Mattsson2008} for the constant-coefficient second derivative; they have the advantage of lower bandwidth and better damping of under-resolved modes. Subsequently, Mattsson \cite{Mattsson2012} extended classical minimum-stencil FD-SBP operators to approximations of the second derivative with variable coefficients---his work represents the current state of the art in classical FD-SBP operators for the second derivative with variable coefficients.

A distinction is made between SBP operators approximating the second derivative that are compatible with the first-derivative SBP operator and those that are not. With appropriate SATs, compatible operators lead to stable semi-discrete forms for partial differential equations (PDEs) that contain cross-derivative terms, such as $\frac{\partial^{2}}{\partial x\partial y}$ \cite{Mattsson2008}. When operators that are not compatible are used to approximate such PDEs, the operators must satisfy additional constraints for an energy estimate to exist. Nevertheless, such operators can always be used, with appropriate SATs, to construct consistent, conservative, and stable semi-discrete forms for PDEs that do not contain cross-derivative terms. This issue was first highlighted by Mattsson, Sv\"ard, and Shoeybi \cite{Mattsson2008}.

The theory of SBP operators has primarily been developed within the context of FD methods, with notable exceptions (see for example \cite{Nordstrom2003,Carpenter1996,Hesthaven1996}). Classical FD-SBP operators were first proposed by Kreiss and Scherer \cite{Kreiss1974} (see Strand for a review of the theory and general solutions \cite{Strand1994}; also see the review papers \cite{Fernandez2014} and \cite{Svard2014}). These SBP operators are constructed to have repeating centered-difference interior operators with biased operators at and near boundary nodes, so that the resultant scheme satisfies the classical definition of an SBP operator. Del Rey Fern\'andez et al.\ \cite{DCDRF2014} have extended the classical FD-SBP theory to operators that have one or more of the following characteristics:  i) non-repeating interior stencil, ii) nonuniform nodal distributions, and iii) exclusion of one or both boundary nodes. Operators having such characteristics are called generalized SBP (GSBP) operators---this terminology has previously been used by Reichert et al.\ \cite{Reichert2011,Reichert2012} to refer to FD methods that relax the definition of an SBP operator (see also \cite{Carpenter1994,Chertock1998,Abarbanel2000,Abarbanel2000b}), but here we use it to indicate the comprehensive generalizations in \cite{DCDRF2014}. Some of the ideas contained within the GSBP framework have individually been discussed by other authors. For example, Carpenter and Gottlieb \cite{Carpenter1996} showed that GSBP operators of maximum degree that include boundary nodes always exist on nearly arbitrary nodal distributions. More recently, Gassner \cite{Gassner2013} has shown that the discontinuous Galerkin collocation spectral element method with Gauss-Lobatto points can be thought of as an SBP-SAT scheme. 

The primary objective of this paper is to investigate the opportunities provided by the GSBP approach to develop efficient operators for the second derivative with variable coefficients. In particular, operators that are more accurate than the application of the first-derivative operator twice. The secondary objective is to further develop the theory of GSBP operators with a repeating interior stencil, as well as classical FD-SBP operators; we do so by proposing a formalism that is a simplification of the work of Kamakoti and Pantano \cite{Kamakoti2009} to easily allow the inclusion of boundary nodes. This formalism leads to a very simple representation of operators with a repeating interior stencil that could be advantageous from an implementation standpoint, both for function evaluations, as well as constructing the Jacobian matrix of implicit methods.

This paper is organized as follows: in Section \ref{Notation}, the notation of the paper is introduced. The first derivative is important for the construction of compatible GSBP operators for the second derivative, so a brief review is given in Section \ref{Generalized SBP operators for the first derivative}. General definitions for non-compatible and compatible GSBP operators for the second derivative are given in Section \ref{Generalized SBP operators for the second derivative: Preliminaries}. In  Section \ref{D2Theory}, we prove that the existence of the constant-coefficient second-derivative GSBP operator guarantees the existence of the variable-coefficient GSBP operator. Section \ref{Generalized SBP operators for the second derivative: Theory for SBP operators with a repeating interior stencil} details additional considerations for GSBP operators with a repeating interior stencil, including classical FD-SBP operators. We present two formulations for operators with a repeating interior stencil, one of which is based on the work in \cite{Mattsson2004b,Mattsson2008, Mattsson2012,Fernandez2012,Fernandez2013,Fernandez2014}, while a more general formulation is constructed by extending the ideas of Kamakoti and Pantano \cite{Kamakoti2009} to include nodes at and near boundaries. Various GSBP and classical FD-SBP operators for the second derivative are constructed in Section \ref{Construction of SBP operators for the second derivative}, including novel GSBP operators on pseudo-spectral nodal distributions and operators that have a repeating interior stencil with variable node-spacing at boundaries---similar in spirit to those developed by Mattsson, Almquist, and Carpenter \cite{Mattsson2014}, but derived by considering the quadrature rules proposed by Alpert \cite{Alpert1999}. These operators are then validated numerically by solving the linear convection-diffusion equation with constant or variable coefficients in Section \ref{Numerical Results}. Finally, conclusions and future work are discussed in Section \ref{Conclusions and future work}.
\section{Notation and definitions}\label{Notation}
The conventions in this paper are based on those laid out in \cite{Hicken2011b,Fernandez2014, DCDRF2014}. GSBP operators refer to SBP operators characterized by one of the following generalizations:  i) non-repeating interior stencil, ii) nonuniform nodal distribution, and iii) exclusion of one or both boundary nodes. On the other hand, the operators originally developed in \cite{Kreiss1974} and \cite{Strand1994} are referred to as classical FD-SBP operators.

Spatial discretization of PDEs can be implemented out using a traditional FD approach where $h$-refinement is carried out by increasing the number of grid nodes. Alternatively, spatial discretization can be implemented using an element approach where the domain is subdivided into a number of elements and $h$-refinement is carried out by increasing the number of elements, with a fixed number of nodes in each element. GSBP operators that have a fixed nodal distribution can only be applied using an element approach.

Vectors are denoted with small bold letters, for example $\mb{x}=[x_{1},\dots,x_{N}]^{T}$, while matrices are presented using capital letters with sans-serif font, for example $\mat{M}$. Capital letters with script type are used to denote continuous functions on a specified domain $x\in[x_{L},x_{R}]$. As an example, $\fnc{U}(x)\in C^{\infty}[x_{L},x_{R}]$ denotes an infinitely differentiable function on the domain $x\in[x_{L},x_{R}]$. Lower case bold font is used to denote the restriction of such functions onto a grid; for example, the restriction of $\fnc{U}$ onto the grid $\mb{x}$ is given by:
\begin{equation}
\mb{u} = \left[\fnc{U}(x_{1}),\dots,\fnc{U}(x_{N})\right]^{T}.
\end{equation}
Vectors with a subscript $h$, for example $\mb{u}_{h}\in \mathbb{R}^{N\times 1}$, represent the solution to a system of discrete or semi-discrete equations.

The restriction of monomials onto a set of nodes is used throughout this paper and is represented by $\mb{x}^{k} = \left[x_{1}^{k},\dots,x_{N}^{k}\right]^{T}$, with the convention that $\mb{x}^{k}=0$ if $k<0$. A superscript is used to denote the order of an operator and a subscript is used to denote which derivative is being approximated. For example, $\mat{D}_{1}^{(p)}$ denotes an SBP approximation to the first derivative of order and degree $p$. The second derivative can be approximated by applying an SBP operator approximating the first derivative twice or by constructing SBP operators that have preferential properties. For this latter type, it is necessary to differentiate between approximations of the constant-coefficient derivative and the variable-coefficient derivative. The convention used is best shown through an example: $\mat{D}_{2}^{(p)}$ represents an order $p$ SBP approximation to the constant-coefficient second derivative, while $\mat{D}_{2}^{(p)}\left(\mat{B}\right)$ represents the approximation to the second derivative with variable coefficients $\fnc{B}$, where $\mat{B} = \mr{diag}[\fnc{B}(x_{1}),\dots,\fnc{B}(x_{N})]$ and $\fnc{B}$ is the variable coefficient. We discuss the degree of SBP operators; that is, the degree of monomial for which they are exact, as well as the order of the operator. The approximation of the derivative has a leading truncation error term for each node, proportional to some power of $h$. The order of the operator is taken as the smallest exponent of $h$ in these truncation errors. The relation between the two for an operator approximating the $m^{\textrm{th}}$ derivative is
\begin{equation}\label{OrderVSdegree}
\textrm{Order} = \textrm{degree}-m+1.
\end{equation}

For GSBP operators with a repeating interior stencil, as well as classical FD-SBP operators, the first and second derivatives are of different order on the interior and near the boundary. In order to differentiate between operators and the various orders, when necessary a superscript is appended to operators for the orders and a subscript is appended to denote which derivative is being approximated. For example, $\mat{D}^{(a,b)}_{i,e}$ denotes the operator for the $i^{\textrm{th}}$ derivative with interior order of $a$ and a minimum order of  $b$ at and near boundary nodes, while the additional subscript $e$ is to differentiate among various versions of the operator. In some cases, one or several of the superscripts are not of interest and are replaced with colons; as an example, $\mat{D}^{(2,:)}_{3}$ denotes an approximation to the third derivative which is of order $2$ on the interior, where the minimum order of nodes near and at the boundary is not specified.

For later use, the $L_{2}$  inner product and norm are defined as
\begin{equation}
\begin{array}{lr}
(\fnc{U},\fnc{V}) = \int_{x_{L}}^{x_{R}}\fnc{U}\fnc{V}\mr{d}x,&||\fnc{U}||^{2} = \int_{x_{L}}^{x_{R}}\fnc{U}^{2}\mr{d}x.
\end{array}
\end{equation}
A discrete inner product and norm have the form
\begin{equation}\label{SBPnorm1}
\begin{array}{lr}
(\mb{u},\mb{v})_{\mat{H}} = \mb{u}^{T}\mat{H}\mb{v},&||\mb{u}||^{2}_{\mat{H}} = \mb{u}^{T}\mat{H}\mb{u},
\end{array}
\end{equation}
where $\mat{H}$ must be symmetric and positive definite.
\section{Generalized SBP operators for the first derivative}\label{Generalized SBP operators for the first derivative}
The set of PDEs of interest here contain both first and second derivatives, and may or may not contain mixed derivatives. For such PDEs, it is possible to construct GSBP operators that lead to stable schemes if certain relationships exist between the first- and second-derivative operators. Therefore, in this section, GSBP operators for the first derivative are briefly reviewed; for classical FD-SBP operators, see the two review papers \cite{Fernandez2014} and \cite{Svard2014}. For more information on GSBP operators for the first derivative, see Del Rey Fern\'andez et al.\ \cite{DCDRF2014}.

To motivate the definition of an SBP operator for the first derivative, consider the unsteady linear-convection equation
\begin{equation}\label{Generalized SBP operators for the first derivative:1}
\frac{\partial \fnc{U}}{\partial t}=-\frac{\partial \fnc{U}}{\partial x},\;x\in[x_{L},x_{R}],\;t\ge0,
\end{equation}
where neither an initial condition nor a boundary condition is specified. The energy method is applied to (\ref{Generalized SBP operators for the first derivative:1}) to construct an estimate on the solution, called an energy estimate, which is then used to determine stability (for more information see \cite{Gustafsson2008,Gustafsson2013,Kreiss2004}). This consists of multiplying the PDE by the solution and integrating in space and transforming the volume integral on the RHS using integration-by-parts. This leads to 
\begin{equation}\label{Generalized SBP operators for the first derivative:2}
\frac{\partial \left\|\fnc{U}\right\|^{2}}{\partial t}=-\left.\fnc{U}^{2}\right|_{x_{L}}^{x_{R}}.
\end{equation}
SBP operators for the first derivative are constructed such that when the energy method is applied to the semi-discrete or fully discrete equations, energy estimates analogous to (\ref{Generalized SBP operators for the first derivative:2}) can be constructed. This leads to the following definition \cite{DCDRF2014}

 \begin{definition}\label{DEFSBPgen}
{\bf Generalized summation-by-parts operator:} A matrix operator $\mat{D}_{1}^{(p)}\in\mathbb{R}^{N\times N}$ is an approximation to the first derivative, on the nodal distribution $\mb{x}$, of order and degree $p$ with the SBP property if
 \begin{enumerate}
 \item $\mat{D}_{1}^{(p)}\mb{x}^{j}= \mat{H}^{-1}\mat{Q} \mb{x}^{j}= j\mb{x}^{j-1}$, $j\in[0,p]$;
 \\
 \item $\mat{H}$, denoted the norm matrix, is symmetric positive definite; and
  \\
 \item $\mat{Q}+\mat{Q}^{T} = \mat{E}$, where $\left(\mb{x}^{i}\right)^{T}\mat{E}\mb{x}^{j}=x_{L}^{i+j}-x_{R}^{i+j}$, $i,j\in[0,r]$, $r\ge p$ and $x_{R}$ and $x_{L}$ are the left and right spatial locations of the boundaries of the block or element.
 \end{enumerate}
  \end{definition}
\noindent The nodal distribution, $\mb{x}$, in Definition \ref{DEFSBPgen} need neither be uniform nor include the boundary nodes. 
 
 Both classical FD-SBP and GSBP operators can be constructed with either a diagonal-norm $\mat{H}$ or a dense-norm $\mat{H}$, where dense norm refers to any $\mat{H}$ that is not diagonal. The matrix $\mat{E}$ is constructed as \cite{DCDRF2014}
  \begin{equation}
  \mat{E} =\mb{t}_{x_{R}}\mb{t}_{x_{R}}^{T}-\mb{t}_{x_{L}}\mb{t}_{x_{L}}^{T}=\mat{E}_{x_{R}}-\mat{E}_{x_{L}}.
  \end{equation}
  The vectors $\mb{t}_{x_{R}}$ and $\mb{t}_{x_{L}}$ satisfy the relations
  \begin{equation}
  \mb{t}_{x_{R}}^{T}\mb{x}^{j} = x_{R}^{j},\;\mb{t}_{x_{L}}^{T}\mb{x}^{j} = x_{L}^{j},\;j\in[0,r],
  \end{equation} 
  and can be thought of as projection operators. This means that $ \mb{t}_{x_{R}}^{T} \mb{u}$ and  $ \mb{t}_{x_{L}}^{T} \mb{u}$ are degree $r$ approximations to $\fnc{U}\left(x_{R}\right)$ and $\fnc{U}\left(x_{L}\right)$, respectively, and are, therefore, of order $r+1$; that is,  
   \begin{equation}\label{degt}
\mb{t}_{x_{R}}^{T}\mb{u}=\fnc{U}(x_{R})+\fnc{O}\left(h^{r+1}\right),\;
\mb{t}_{x_{L}}^{T}\mb{u}=\fnc{U}(x_{L})+\fnc{O}\left(h^{r+1}\right),
 \end{equation}
 where $\mb{u}$ is the projection of $\fnc{U}$ onto the nodal distribution. In (\ref{degt}), $h$ is the spacing between nodes for GSBP operators applied using the traditional approach, while for GSBP operators applied using the element approach, $h$ is some measure of the spacing between nodes, for example the average spacing.

 The semi-discrete representation of (\ref{Generalized SBP operators for the first derivative:1}) using GSBP operators is  
 \begin{equation}\label{Generalized SBP operators for the first derivative:3}
 \frac{\mr{d}\mb{u}_{h}}{\mr{d}x}=-\mat{D}_{1}^{(p)}\mb{u}_{h},
 \end{equation}
 where no boundary or initial conditions are imposed. The energy method consists of multiplying (\ref{Generalized SBP operators for the first derivative:3}) by $\mb{u}_{h}^{T}\mat{H}$ and adding the transpose of the product, which gives
 \begin{equation}\label{Generalized SBP operators for the first derivative:4}
 \frac{\mr{d}\left\|\mb{u}_{h}\right\|_{\mat{H}}^{2}}{\mr{d}x}=-\mb{u}_{h}^{T}\left[\mat{H}\mat{D}_{1}^{(p)}+\left(\mat{D}_{1}^{(p)}\right)^{T}\mat{H}\right]\mb{u}_{h}.
 \end{equation}
 \noindent Using Definition (\ref{DEFSBPgen}) results in 
  \begin{equation}\label{Generalized SBP operators for the first derivative:5}
 \frac{\mr{d}\left\|\mb{u}_{h}\right\|_{\mat{H}}^{2}}{\mr{d}x}=-\mb{u}_{h}^{T}\mat{E}\mb{u}_{h}=-\left(\tilde{u}_{x_{R}}^{2}-\tilde{u}_{x_{L}}^{2}\right),
 \end{equation} 
 where $\tilde{u}_{x_{R}}=\mb{t}_{x_{R}}^{T}\mb{u}_{h}$ and $\tilde{u}_{x_{L}}=\mb{t}_{x_{L}}^{T}\mb{u}_{h}$, and it can be seen that (\ref{Generalized SBP operators for the first derivative:5}) is a discrete analogue of (\ref{Generalized SBP operators for the first derivative:2}).
 \vspace{2cm}
\section{Generalized SBP operators for the second derivative}\label{Generalized SBP operators for the second derivative}
\subsection{Preliminaries}\label{Generalized SBP operators for the second derivative: Preliminaries}
In this section, the definition of classical FD-SBP operators approximating the second derivative, given by \cite{Mattsson2004b,Mattsson2008,Mattsson2012}, is extended to accommodate the derivation of GSBP operators. The form we propose combines ideas from \cite{Mattsson2004b,Mattsson2008,Mattsson2012}, as well as our extension of the ideas of Kamakoti and Pantano \cite{Kamakoti2009} on the interior stencil of classical FD-SBP operators (see Section \ref{Generalized SBP operators for the second derivative: Theory for SBP operators with a repeating interior stencil}).

The equations that an operator must satisfy in order to approximate the second derivative with a variable coefficient, denoted the degree equations, are based on monomials restricted onto the nodes of the grid. Given that the operator must approximate {\footnotesize $\frac{\partial}{\partial x}\left(\fnc{B}\frac{\partial \fnc{U}}{\partial x}\right)$}, it is necessary to determine what degree monomial to insert for {\small$\fnc{B}$} and {\small$\fnc{U}$} in constructing the degree equations. Taking {\small$\fnc{B}=x^{k}$} and {\small$\fnc{U}=x^{s}$} and inserting into the second derivative gives
\begin{equation}
\frac{\partial}{\partial x}\left(x^{k}\frac{\partial x^{s}}{\partial x}\right)=s(k+s-1)x^{k+s-2}.
\end{equation}
To be of order $p$, second-derivative operators must be of degree $p+1$ (from \ref{OrderVSdegree}). This implies that all combinations of $k+s\leq p+1$ must be satisfied. Thus, the degree equations have the following form:
\begin{equation}\label{accD2V}
\mat{D}_{2}^{(p)}\left(\textrm{diag}\left(\mb{x}^{k}\right)\right)\mb{x}^{s}=s(k+s-1)\mb{x}^{k+s-2},\: k+s\leq p+1,
\end{equation}
where $\textrm{diag}\left(\mb{x}^{k}\right)$ is a diagonal matrix such that the $i^{\textrm{th}}$ diagonal entry is the $i^{\textrm{th}}$ entry of $\mb{x}^{k}$. If there are $N$ nodes in the nodal distribution, then each combination of $k+s$ in (\ref{accD2V}) returns a vector of $N$ equations.

The maximum attainable degree and order for an operator for the second derivative are given by the following lemma:
\begin{lemma}\label{MaxdegD2}
An operator, $\mat{D}_{2}^{(p)}\in\mathbb{R}^{N\times N}$, is at most of order $p\leq N-2$ and degree $N-1$.
\begin{proof}
Consider the degree equations for a constant-coefficient operator:
\begin{equation}
\mat{D}_{2}^{(p)}\mb{x}^{k}=k\left(k-1\right)\mb{x}^{k-2},\;k\in[0,p+1].
\end{equation}
Taking $p=N-2$, the degree equations can be recast as
\begin{equation}
\mat{D}_{2}^{(N-2)}\mat{X}=\tilde{\mat{X}},
\end{equation}
where $\mat{X}=\left[\mb{x}^{0},\dots,\mb{x}^{N-1}\right]$ and $\tilde{\mat{X}}=\left[\mb{0},\mb{0},2\mb{x}^{0},\dots,N\left(N-1\right)\mb{x}^{N-2}\right]$. The matrix $\mat{X}$ is the Vandermond matrix and is invertible, where the columns of $\mat{X}$ represent a basis for $\mathbb{R}^{N\times N}$. Therefore, a unique solution exists, given as $\mat{D}_{2}=\tilde{\mat{X}}\mat{X}^{-1}$, and by examining the range of the operator $\mat{D}_{2}$, i.e., $\tilde{\mat{X}}$, it is clear that $\mat{D}_{2}$ is of most degree $N-1$ and hence order $p=N-2$.
\end{proof}
\end{lemma}

An immediate consequence of Lemma \ref{MaxdegD2} is the following corollary:

\begin{corollary}\label{MaxdegD2B}
An operator $\mat{D}_{2}^{(p)}\left(\mat{B}\right)\in\mathbb{R}^{N\times N}$, approximating the second derivative with variable coefficients, is at most of order $p=N-2$ and degree $N-1$.
\begin{proof}
The set of equations for the constant-coefficient case is a subset of the equations for the variable-coefficient operator, and therefore, by Lemma \ref{MaxdegD2} $\mat{D}_{2}^{(p)}\left(\mat{B}\right)$ is, at best, of order $N-2$.
\end{proof}
\end{corollary}
\subsection{GSBP operators for the second derivative}\label{D2Theory}

For classical FD-SBP operators, one of the drawbacks of the application of the first-derivative operator twice is that the interior stencil uses nearly twice as many nodes as minimum-stencil operators. For GSBP operators that can only be applied using an element approach, no such concept exists. Regardless, the application of the first-derivative operator twice results in an approximation that is of lower order than the first-derivative operator. Thus, in general, we search for GSBP operators approximating the second derivative that match the order of the first-derivative operator; such operators are denoted order matched. These ideas lead to the following definition:
\begin{definition}\label{DefD2}
{\bf Order-matched second-derivative GSBP operator:} The matrix $\mat{D}^{(p)}_{2}(\mat{B})\in \mathbb{R}^{N\times N}$ is  a GSBP operator approximating the second derivative, $\frac{\partial}{\partial x}\left(\fnc{B}\frac{\partial \fnc{U}}{\partial x}\right)$, of degree $p+1$ and order $p$ that is order matched to the GSBP operator $\mat{D}_{1}^{(p)}=\mat{H}^{-1}\mat{Q}$, on a nodal distribution $\mb{x}$, if if satisfies the equations
\begin{equation}
\mat{D}_{2}^{(p)}\left(\textrm{diag}\left(\mb{x}^{k}\right)\right)\mb{x}^{s}=s(k+s-1)\mb{x}^{k+s-2},\: k+s\leq p+1,
\end{equation}
 and is of the form
\begin{equation}
 \mat{D}^{(p)}_{2}(\mat{B})=\mat{H}^{-1}\left\{-\mat{M}\left(\mat{B}\right)+\mat{E}\mat{B}\mat{D}_{b}^{(\ge p+1)}\right\},
\end{equation}
 where
 \begin{equation}
 \mat{M}(\mat{B}) = \sum_{i=1}^{N}\mat{B}(i,i)\mat{M}_{i}.
 \end{equation}
The matrices $\mat{M}_{i}$, $\mat{B}$, and $\mat{D}_{b}^{(\ge p+1)}$ are $\in\mathbb{R}^{N\times N}$, $\mat{M}_{i}$ is symmetric positive semi-definite, 
$$\mat{B} = \textrm{diag}(\fnc{B}(x_{1}),\dots,\fnc{B}(x_{N})),$$ 
and $\mat{D}_{b}^{(\ge p+1)}$ is an approximation to the first derivative of degree and order $\ge p+1$.
\end{definition}

If one takes $\mat{B}$ to be the identity matrix, then Definition \ref{DefD2} collapses onto that given by Mattsson and Nordstr\"om \cite{Mattsson2004b} for classical FD-SBP operators---defining the relevant matrix in their definition as the sum of the $\mat{M}_{i}$---where we do not specify further restrictions on the form of the $\mat{M}_{i}$ in order to allow for GSBP operators. The extension to variable coefficients, by taking the sum of matrices multiplied by the variable coefficients, is an extension and simplification of the work by Kamakoti and Pantano, who decompose the internal stencil of FD approximations to the second derivative with variable coefficients as the sum of the variable coefficient multiplying a third-order tensor. Definition \ref{DefD2} can be applied to dense-norm GSBP operators, though we do not pursue this further in this paper.

Definition \ref{DefD2} is sufficient to derive energy estimates, with appropriate SATs, for PDEs that do not contain mixed-derivative terms. Without additional constraints, however, it does not guarantee that an energy estimate exists for PDEs with cross-derivative terms. Further restrictions need to be applied to Definition \ref{DefD2} such that an energy estimate exists. One possibility is what is referred to as compatible operators \cite{Mattsson2008}; these operators are guaranteed to produce energy estimates, again with appropriate SATs, for PDEs with cross-derivative terms. These ideas lead to the following definition:

\begin{definition}\label{DefD2C}
{\bf Order-matched compatible second-derivative GSBP operator:}
A diagonal-norm order-matched GSBP operator, $\mat{D}^{(p)}_{2}(\mat{B})\in \mathbb{R}^{N\times N}$, for the second derivative, is compatible with the first-derivative GSBP operator, $\mat{D}_{1}^{(p)}$, if in addition to the requirements of Definition \ref{DefD2}, 
\begin{equation}
 \mat{M}\left(\mat{B}\right)=\left(\mat{D}_{1}^{(p)}\right)^{T}\mat{H}\mat{B}\mat{D}_{1}^{(p)}+\mat{R}\left(\mat{B}\right),
 \end{equation}
  where
 \begin{equation}
 \mat{R}\left(\mat{B}\right)=\sum_{i=1}^{N}\mat{B}(i,i)\mat{R}_{i},
 \end{equation}
 where $\mat{R}_{i}$ is symmetric positive semi-definite.
\end{definition}

The idea of decomposing the operator as the application of the first-derivative operator twice plus a corrective term was first proposed by Mattsson et al.\ \cite{Mattsson2008} and later used by Mattsson \cite{Mattsson2012} to construct classical FD-SBP operators to approximate the second derivative with variable coefficients. The definition of compatible operators is limited to diagonal-norm operators; for the variable-coefficient case, it is unclear how to derive energy estimates for dense-norm operators (see Mattsson and Almquist \cite{Mattsson2013} for a discussion and potential solution).
 
The compatibility that is necessary is between the first-derivative operators approximating the mixed derivatives and the second-derivative operator. In addition, an energy estimate is guaranteed to exist, with appropriate SATs, if the norms of all operators are the same. In practice, this means that all first-derivative terms are typically approximated using the same GSBP operator.

An order-matched and compatible $\mat{D}_{2}^{(p)}\left(\mat{B}\right)$ SBP operator, as given in Definition \ref{DefD2C}, is the application of the first-derivative operator twice plus a corrective term to increase the degree of the resultant operator. The application of the first-derivative operator twice already satisfies a number of the degree equations (\ref{accD2V}), and the corrective term is added such that the remaining degree equations in order for the operator to be of order $p$ are satisfied, while continuing to satisfy those degree equations satisfied by the application of the first-derivative operator twice. Applying an order $p$ SBP operator for the first derivative twice results in 
\begin{equation}\label{accapp2}
\begin{array}{lcr}
\mat{D}_{1}^{(p)}\textrm{diag}\left(\mb{x}^{k}\right)\mat{D}_{1}^{(p)}\mb{x}^{s}=s\mat{D}_{1}^{(p)}\mb{x}^{s+k-1}&=&s(s+k-1)\mb{x}^{s+k-2},\\\\
&&s+k\leq p+1,\;s\leq p.
\end{array}
\end{equation}
Equations (\ref{accapp2}) show that only the equations for $s=p+1,\; k=0$  are not satisfied by the application of the first-derivative operator twice. 
We now use the observation that the application of the first-derivative operator twice already satisfies a number of the degree equations to propose a construction of order-matched GSBP operators for the second derivative with variable coefficients as the application of the first-derivative operator twice plus a corrective term that is modelled after the constant-coefficient operator.
\begin{theorem}\label{CorollaryTheorem2}
The existence of a diagonal-norm compatible and order-matched GSBP operator $\mat{D}_{2}^{(p)}$ of order $p$ and degree $p+1$ is sufficient for the existence of a compatible and order-matched GSBP operator $\mat{D}_{2}^{(p)}\left(\mat{B}\right)$, for $p+1\leq N-1$ and $N\ge3$.  
\end{theorem}
\begin{proof}
Consider constructing the operator as 
\begin{equation}\label{FormD2B}
\mat{H}^{-1}\left[-\left(\mat{D}_{1}^{(p)}\right)^{T}\mat{H}\mat{B}\mat{D}_{1}^{(p)}-\frac{\sum_{i=1}^{n}b_{i}}{n}\mat{R}_{c}+\mat{E}\mat{B}\mat{D}_{b}^{(\ge p+1)}\right],
\end{equation}
where $\mat{R}_{c}$ and $\mat{D}_{b}^{(\ge p+1)}$ are from the constant-coefficient operator. As has been argued, the additional equations that must be satisfied are for $(k,s)=(0,p+1)$. Since (\ref{FormD2B}) collapses onto the constant-coefficient operator for this condition, it automatically satisfies these additional equations. What remains to be shown is that the remaining degree equations are still satisfied.

Now $\mat{D}_{b}^{(\ge p+1)}=\mat{D}_{1}^{(p)}+\mat{A}$, where $\mat{A}$ is a corrective term such that $\mat{D}_{b}^{(\ge p+1)}$ is at least one order more accurate than the first-derivative operator. The application of the first-derivative operator twice can decomposed into
\begin{equation}\label{FormD2B2}
\mat{D}_{1}^{(p)}\mat{B}\mat{D}_{1}^{(p)}=\mat{H}^{-1}\left[-\left(\mat{D}_{1}^{(p)}\right)^{T}\mat{H}\mat{B}\mat{D}_{1}^{(p)}+\mat{E}\mat{B}\mat{D}_{1}^{(p)}\right].
\end{equation}
Therefore (\ref{FormD2B}) can be recast as
\begin{equation}\label{FormD2B3}
\mat{D}_{1}^{(p)}\mat{B}\mat{D}_{1}^{(p)}+\mat{H}^{-1}\left\{-\frac{\sum_{i=1}^{n}b_{i}}{n}\mat{R}_{c}+\mat{E}\mat{B}\mat{A}\right\}.
\end{equation}

\noindent Examining the constant-coefficient version of (\ref{FormD2B}), it can be seen that both $\mat{A}$ and $\mat{R}_{c}$ must be be zero for $\mb{x}^{s}$ for $s\leq p$. Therefore, we have proven that (\ref{FormD2B}) leads to a compatible and order-matched GSBP operator for the second derivative with variable coefficients. 
\end{proof}
%
%

The implication of Theorem \ref{CorollaryTheorem2} is that the search for order-matched GSBP operators for the variable-coefficient case reduces to the search for order-matched GSBP operators for the constant-coefficient case. This substantially simplifies both the proof that order-matched GSBP operators exist for a given nodal distribution and their construction. Consider trying to solve the non-linear equations for a $13$-node GSBP operator; for such an operator, the eigenvalue problem of $13$ matrices, of size $13\times 13$, must be solved. In practice, one solves the constant-coefficient problem and if such operators exist, Theorem (\ref{CorollaryTheorem2}) guarantees that order-matched GSBP operators exist for the variable-coefficient case. Moreover, Theorem (\ref{CorollaryTheorem2}) gives a simple means of constructing order-matched GSBP operators from the constant-coefficient order-matched GSBP operators. The implications of this will be further discussed in Section \ref{Construction of SBP operators for the second derivative}.
\subsection{GSBP operators with a repeating interior stencil}\label{Generalized SBP operators for the second derivative: Theory for SBP operators with a repeating interior stencil}
The focus of this section is on compatible order-matched operators with a repeating interior stencil. The repeating interior stencil requires satisfying additional constraints and thus further specifies the form of $\mat{R}\left(\mat{B}\right)$. Here we present two versions: one based on ideas in \cite{Mattsson2004b,Mattsson2008, Mattsson2012,Fernandez2012,Fernandez2013,Fernandez2014} and another based on a simplification of the ideas of Kamokoti and Pantano \cite{Kamakoti2009} that allows for a simple extension to operators that include nodes near and potentially at boundaries. The second form of the operator is not only convenient for analysis, but also from an implementation standpoint. It reduces the application of the operator to one loop. This form is also advantageous for the construction of implicit methods that require the linearization of the order-matched GSBP operator, since the linearization is completely transparent. Moreover, it is a convenient formalism for presenting particular instances of operators. The first form, which corresponds to classical minimum-stencil FD-SBP operators, is given as
\begin{equation}\label{SBPD2Form}
\begin{array}{rcl}
        	\mat{D}_{2}^{(2p,p)} (\mat{B})&=&\mat{H}^{-1}\left[-\left(\mat{D}_{1}^{(2p,p)}\right)^{T}\mat{H}\mat{B}\mat{D}_{1}^{(2p,p)}+\mat{E}\mat{B}\mat{D}_{1}^{(:,\ge p+1,:)}\right]\\\\
	&&-\frac{1}{h}\mat{H}^{-1}\sum_{i=p+1}^{2p}\alpha^{(p)}_{i}\left(\mat{\tilde{D}}_{i,p}^{(2,1,:)}\right)^{T}\mat{C}_{i}^{(p)}\mat{B}\mat{\tilde{D}}_{i,p}^{(2,1,:)}.
	\end{array}
\end{equation}
The $\mat{\tilde{D}}_{i,p}^{(2,1,:)}$ operators have an interior stencil that spans $2p+1$ nodes, while biased stencils at the first $2p$ and last $2p$ nodes use $3p$ nodes starting at either boundary. The interior stencil is a second-order centered-difference approximation to the $i^{\textrm{th}}$ derivative, while the biased stencils are first-order accurate. The tilde notation denotes an undivided difference approximation. Constructed as such, the corrective term, multiplied by $\mat{H}$, is guaranteed to be negative semi-definite, as long as the $\mat{C}_{i}^{(p)}$, which are diagonal matrices of the form $\mat{C}_{i}^{(p)} = \diag\left(c_{i,1}^{(p)},\dots,c_{i,2p}^{(p)},1,\dots,1,c_{i,2p}^{(p)},\dots,c_{i,1}^{(p)}\right)$, are positive semi-definite. The operator $\mat{D}_{1}^{(:,\ge p+1)}$ is an approximation to the first derivative of at least order $p+1$. The $\alpha^{(p)}_{i}$ coefficients are
\begin{itemize}
\item $p=1$: $\alpha_{2}^{(1)}=\frac{1}{4}$; 
\item $p=2$: $\alpha_{3}^{(2)}=\frac{1}{18}$, $\alpha_{4}^{(2)}=\frac{1}{48}$;
\item $p=3$: $\alpha_{4}^{(3)}=\frac{1}{80}$, $\alpha_{5}^{(3)}=\frac{1}{100}$, $\alpha_{6}^{(3)}=\frac{1}{720}$; and
\item $p=4$: $\alpha_{5}^{(4)}=\frac{1}{350}$, $\alpha_{6}^{(4)}=\frac{1}{252}$, $\alpha_{7}^{(4)}=\frac{1}{980}$, $\alpha_{8}^{(4)}=\frac{1}{11200}$.
\end{itemize}
We note that for $p>4$, the first-derivative operator requires more than $2p$ nodes that do not have the interior stencil (see \cite{Albin2014}), so the above definitions would need to be changed accordingly. The form of the various constituent matrices is different than that proposed by Mattsson \cite{Mattsson2012}. Regardless, both formulations lead to the same interior stencil and the analysis in this section applies to both formulations. From (\ref{SBPD2Form}) it can be seen that $\mat{R}\left(\mat{B}\right)$ corresponding to Definition (\ref{DefD2C}) is given as
\begin{equation}\label{classicalR}
\mat{R}\left(\mat{B}\right)=-\frac{1}{h}\sum_{i=p+1}^{2p}\alpha^{(p)}_{i}\left(\mat{\tilde{D}}_{i,p}^{(2,1,:)}\right)^{T}\mat{C}_{i}^{(p)}\mat{B}\mat{\tilde{D}}_{i,p}^{(2,1,:)}.
\end{equation}
With form (\ref{classicalR}) and the $\alpha_{i}^{(p)}$s specified above, the interior stencil of $\mat{D}_{2}^{(2p,p,:)}(\mat{B})$ is fully specified. As is discussed in Section \ref{Construction of SBP operators for the second derivative}, this form can be used to construct a subset of order-matched GSBP operators with a repeating interior stencil.

The second form is given as 
\begin{equation}\label{D2ALT}
\mat{D}_{2}^{(2p,p)}\mat{H}(\mat{B}) = \mat{H}^{-1}\left[\sum_{i=1}^{N}\mat{B}\left(i,i\right)\left(\mat{M}_{\mat{D},i}+\mat{R}_{i}\right)+\mat{E}\mat{B}\mat{D}_{1}^{(\ge p+1)}\right],
\end{equation}
where
\begin{equation}
\left(\mat{D}_{1}^{(2p,p)}\right)^{T}\mat{H}\mat{B}\mat{D}_{1}^{(2p,p)} = \sum_{i=1}^{N}\mat{B}\left(i,i\right)\mat{M}_{\mat{D},i},
\end{equation}
and all of the matrices are $\in\mathbb{R}^{N\times N}$. Clearly, the first form can be recast in the second form. Here we concentrate on GSBP operators that have the same internal stencil as (\ref{SBPD2Form}); however, form (\ref{D2ALT}) easily allows contemplating other types of internal stencils. Form (\ref{D2ALT}) can be further collapsed by retaining only the nonzero blocks, such that the operators, applied to a vector $\mb{u}$, can be constructed as
\begin{equation}\label{D2ALT2}
\begin{array}{rcl}
\mat{D}_{2}\left(\mat{B}\right)^{(2p,p)}\mb{u} &=& \sum_{i=1}^{g}\mat{B}(i,i)\mat{M}_{i}\mb{u}_{i}+\mat{P}_{i}\mat{M}_{i}\mat{P}_{i}\mb{u}_{n-i+1}\\\\
&&+ \sum_{i=g+1}^{N-g}\mat{B}(i,i)\mat{M}_{INT}\mb{u}(i-p:i+p).
\end{array}
\end{equation}
The organization of (\ref{D2ALT2}) is based on the variable coefficients rather than the vector $\mb{u}$. The matrices $\mat{M}_{i}$, $i\in[1,g]$ originate from the contributions of the biased stencils near and at the boundary. The matrix $\mat{M}_{INT}$ represents the contributions from the interior stencil and is of size $\left(2p+1\right)\times\left(2p+1\right)$. The notation $\mb{u}_{i}$ is to denote that a particular portion of the vector $\mb{u}$ is being used, while $\mb{u}(i-p:i+p)$ is to be taken literally. We have also taken advantage of the fact that the operators we are interested in have biased stencils that are the same under a change in the direction of the axis. Therefore, the $\mat{M}$ associated with nodes near and at the right boundary are the permutation of the rows and columns of the same matrices from the left boundary. The operation of taking the permutation of the rows and columns of a matrix is easily performed by pre- and post-multiplying by a permutation matrix with ones along its anti-diagonal, denoted here by the matrices $\mat{P_{i}}$.

We discuss the difficulties presented by (\ref{D2ALT2}) in deriving compatible and order-matched GSBP and classical minimum-stencil FD-SBP operators in Section \ref{ConstructionRepeat}. Here we are interested in the internal stencil. First, some properties of $\mat{M}_{INT}$ from (\ref{SBPD2Form}) are summarized in the following proposition:
\begin{proposition}\label{CharMINT}
The matrix $\mat{M}_{INT}$ of a minimum-stencil compatible and order-matched GSBP operator, $\mat{D}_{2}\left(\mat{B}\right)^{(2p,p)}$, constructed from (\ref{SBPD2Form}), has the following properties:
\begin{itemize}
	\item it is of size $\left(2p+1\right)\times\left(2p+1\right)$;
	\item the $j^{\textrm{th}}$ coefficient of the internal stencil is given as the sum of the $j^{\textrm{th}}$ off diagonal multiplied by the corresponding variable coefficient, with the convention that the main diagonal is zero and those to the right are enumerated using positive numbers and those to the left are enumerated using negative numbers;
	\item it is bisymmetric;
	\item the entries in the upper right-hand triangle with corner entries $(1,2p+1)$, $(1,p+2)$, and $(2p+1,p)$, and the associated lower left-hand triangle, have entries that are zero; and
	\item all rows and columns sum to zero.
\end{itemize}
\end{proposition}

The properties listed in Proposition \ref{CharMINT} can be easily observed by expanding (\ref{SBPD2Form}). As an example, consider $\mat{M}_{INT}$ for the classical minimum-stencil FD-SBP operator with $p=2$, which using (\ref{SBPD2Form}), is 
\begin{equation}
\mat{M}_{INT}=\left[ \begin {array}{ccccc} -\frac{1}{24}&\frac{1}{6}&-\frac{1}{8}&0&0\\ \noalign{\medskip}\frac{1}
{6}&-\frac{5}{6}&\frac{1}{2}&\frac{1}{6}&0\\ \noalign{\medskip}-\frac{1}{8}&\frac{1}{2}&-\frac{3}{4}&\frac{1}{2}&-\frac{1}{8}
\\ \noalign{\medskip}0&\frac{1}{6}&\frac{1}{2}&-\frac{5}{6}&\frac{1}{6}\\ \noalign{\medskip}0&0&-\frac{1}{8}&\frac{1}
{6}&-\frac{1}{24}\end {array} \right], 
\end{equation}
while the application of the first-derivative operator twice gives
\begin{equation}
\mat{M}_{INT}=\left[ \begin {array}{ccccc} -{\frac {1}{144}}&\frac{1}{18}&0&-\frac{1}{18}&{\frac {1
}{144}}\\ \noalign{\medskip}\frac{1}{18}&-\frac{4}{9}&0&\frac{4}{9}&-\frac{1}{18}\\ \noalign{\medskip}0
&0&0&0&0\\ \noalign{\medskip}-\frac{1}{18}&\frac{4}{9}&0&-\frac{4}{9}&\frac{1}{18}
\\ \noalign{\medskip}{\frac {1}{144}}&-\frac{1}{18}&0&\frac{1}{18}&-{\frac {1}{144}}
\end {array} \right].
\end{equation}

For classical FD-SBP operators approximating the second derivative with constant coefficients, one of the motivations for minimum-stencil operators \cite{Mattsson2004b,Mattsson2008} is that besides increasing the degree of the operator, the resulting interior stencil has smaller bandwidth than the application of the first-derivative operator twice. For example, taking $p=2$, the application of the first-derivative operator twice has an interior operator
\begin{equation}
\frac{1}{h^2}\left[ \begin {array}{ccccccccc} {\frac {1}{144}}&-\frac{1}{9}&\frac{4}{9}&\frac{1}{9}&-{
\frac {65}{72}}&\frac{1}{9}&\frac{4}{9}&-\frac{1}{9}&{\frac {1}{144}}\end {array} \right],
\end{equation}
and in general has $4p+1$ coefficients. On the other hand, the minimum-stencil operator has an interior operator given as
\begin{equation}
\frac{1}{h^2}\left[ \begin {array}{ccccc} -\frac{1}{12}&\frac{4}{3}&-\frac{5}{2}&\frac{4}{3}&-\frac{1}{12}\end {array}
 \right], 
\end{equation}
and in general has $2p+1$ coefficients. Thus, for constant-coefficient operators, the minimum-stencil requires fewer floating-point operations for interior nodes, as compared to the application of the first-derivative operator twice. 

The form of compatible classical FD-SBP operators for the second derivative with constant coefficients motivated the construction of the variable-coefficient operator. Thus, the stencil width has the same properties as the constant-coefficient case; that is, the application of the first-derivative operator twice uses $4p+1$ nodes, while the minimum-stencil operator uses $2p+1$ nodes. However, the number of nodes hides a paradoxical fact: for $p<4$, the interior stencil of the compatible classical minimum-stencil FD-SBP operator, $\mat{D}_{2}^{(p)}\left(\mat{B}\right)$, requires a greater number of operations to construct as compared to the application of the first-derivative operator twice. Table \ref{Nonzero} summarizes the number of nonzero entries in $\mat{M}_{INT}$ and is, therefore, reflective of the number of floating-point operations necessary for constructing the interior stencil (the last row is constructed from the observable pattern in $p\in[1,4]$ and suggests that for $p>3$, the minimum-stencil operator has fewer nonzero entries). Nevertheless, for $p<3$, minimum-stencil operators are more accurate than the application of the first-derivative operator twice. Moreover, they can be advantageous when used for problems where the Jacobian matrix must be constructed, for example in optimization. For operators with a repeating interior stencil, this results from the fact that the minimum-stencil operators have smaller bandwidth than the application of the first-derivative operator twice, and therefore leads to Jacobian matrices that require less storage and fewer floating-point operations to invert---this issue does not apply to element-based operators.
\begin{table}[t!]
\caption{The number of nonzero entries in $\mat{M}_{INT}$}
\begin{center}
\begin{tabular}{c|c|c}
$p$&$\mat{D}_{2}(\mat{B})$&$\mat{D}_{1}\mat{B}\mat{D}_{1}$\\
\hline
1&7&4\\
2&19&16\\
3&37&36\\
4&61&64\\
p&$(2p+1)+2\sum_{i=1}^{p}p+i$&$4p^{2}$
\end{tabular}
\end{center}
\label{Nonzero}
\end{table}%

\section{Construction of GSBP operators for the second derivative}\label{Construction of SBP operators for the second derivative}
\subsection{Preliminaries}
For compatible order-matched GSBP operators it is necessary to first solve for the first-derivative GSBP operator, the degree equations for which are given by
\begin{equation}\label{accuracy2}
\begin{array}{lr}
\mat{Q}\mb{x}^{j}=j\mat{H}\mb{x}^{j-1},&j\in[0,p].
\end{array}
\end{equation}
The solution to the degree equations (\ref{accD2V}) and (\ref{accuracy2}) typically results in free parameters that must be specified. This naturally leads to the concept of optimization. The GSBP norm matrix is an approximation to the $L_{2}$ inner product \cite{DCDRF2014} and is used to compute the error in simulations, as well as functionals of the solution. Therefore, here we use the discrete GSBP $L_{2}$ inner product of the error as the objective function, which for the first-derivative operator, $\mat{D}_{1}^{(p)}$, is given as 
\begin{equation}\label{optimaD1D}
J_{p+1}= \mb{e}_{p+1}^{T}\mat{H}\mb{e}_{p+1},
 \end{equation}
 where the error vector is given as
 \begin{equation}
  \mb{e}_{p+1}=\mat{D}_{1}^{(p)}\mb{x}^{p+1}-\left(p+1\right)\mb{x}^{p}.
 \end{equation}
 For GSBP operators approximating the second derivative with variable coefficients, there are several error vectors, each of which is given by
\begin{equation}
\mb{e}_{k,s} = \mat{D}_{2}\left(\textrm{diag}\left(\mb{x}^{k}\right)\right)\mb{x}^{s}-s(s+k-1)\mb{x}^{s+k-2},
\end{equation}
and the objective function is constructed as
\begin{equation}\label{objD2}
J_{p+2}=\sum_{i=0}^{p+2}\left(\mb{e}_{i,p+2-i}\right)^{T}\mat{H}\mb{e}_{i,p+2-i}.
\end{equation}
\subsection{Finite-difference SBP and hybrid Gauss-trapezoidal operators}\label{ConstructionRepeat}
$\textrm{ }$

The generality of form (\ref{D2ALT}), and hence (\ref{D2ALT2}), results in a large system of nonlinear equations for the positive semi-definite requirement on $\mat{R}$. Without further simplification, Maple is unable to find solutions; how to facilitate solutions of the more general form is an ongoing area of research. Thus, in this paper we solve form (\ref{SBPD2Form}); however, form (\ref{D2ALT2}) is extremely convenient both for implementing the operators in a computer code, as well as presenting them.

In addition to classical FD-SBP operators, we also examine operators with a repeating interior stencil that have a number of nodes near the boundaries that are not equally spaced. This idea was first proposed by Mattsson et al.\ \cite{Mattsson2014}. These operators have some very attractive properties; the magnitude of the truncation terms at the boundaries scales as some power of the mesh spacing. By allowing the nodal spacing to vary near the boundaries, it is possible to reduce the magnitude of the error originating from the biased stencils near and at the boundaries \cite{Mattsson2014}. For diagonal-norm classical FD-SBP operators this is very attractive, as the order of the operator reduces by half at these nodes.

Deriving the optimal nodal locations beyond two or three nodes, while at the same time ensuring that a positive-definite norm matrix can be found,  becomes an increasingly difficult task \cite{Mattsson2014}. Alternatively, Del Rey Fern\'andez et al.\ \cite{DCDRF2014} have proven that the norm matrix of a GSBP operator is associated with a quadrature rule of certain degree. Therefore, the search for diagonal-norm GSBP operators reduces to the search for quadrature rules with positive weights. If a quadrature rule with the required properties exists, then it is possible to simply use the nodal distribution from that quadrature rule. Such quadrature rules, denoted hybrid Gauss-trapezoidal quadrature rules, were proposed by Alpert \cite{Alpert1999}. The nodal locations and quadrature weights are derived from the solution to

 \begin{equation}\label{Alpert}
 \sum_{i=1}^{j}\tilde{w}_{i}\tilde{x}_{i}^{r}=\frac{B_{r+1}\left(a\right)}{r+1},\quad r = 0,1,\dots,2j-2,
 \end{equation}
 where $B_{i}\left(x\right)$ is the $i^{\textrm{th}}$ Bernoulli polynomial and $B_{0}\left(x\right)=1$. The parameters $a$ and $j$ are chosen so that a particular degree is attained. If they are chosen such that $a=j$, then it is possible to show that the resultant quadrature rule is positive definite up to degree $20$ \cite{Alpert1999}, which is the approach taken here. In order to include the boundary nodes, the additional equation $\tilde{x}_{0}=0$ is added to the system (\ref{Alpert}). Therefore, it is possible to construct GSBP operators that do or do not include boundary nodes. To differentiate between the two possibilities, we denote hybrid Gauss-trapezoidal Lobatto as GSBP operators that include the boundary nodes and hybrid Gauss-trapezoidal as those that do not. To construct a nodal distribution on $x\in[0,1]$, the following relations are used:
\begin{equation}
\begin{array}{l}
x_{i}=h\tilde{x}_{i},\quad x_{N-(i-1)}=1-h\tilde{x}_{i},\quad i\in[0,j],\\\\
x_{i}=h(a+i),\quad i\in[0,n-1],
\end{array}
\end{equation} 
 where $h = \frac{1}{n+2a-1}$, $n$ is the number of uniformly distributed nodes, and the total number of nodes is given as $N = n+2j$. 
 
 We summarize the steps taken to construct both classical FD-SBP operators and the hybrid Gauss-trapezoidal-Lobatto and hybrid Gauss-trapezoidal operators:
\begin{itemize}
    \item solve the degree equations (\ref{accuracy2}) for the first-derivative GSBP operator;
    \item if there are free parameters, optimize using (\ref{optimaD1D});
    \item if any free parameters remain, set them to zero;
  	\item construct the GSBP operator for the second derivative using (\ref{SBPD2Form});
	\item the degree equations, (\ref{accD2V}), are formed and the first $2p$ are solved;
	\item typically, this results in families of solutions, each with free parameters;
	\item free parameters are specified through optimization, using the objective function (\ref{objD2}), with the constraint that the $\mat{C}$ matrices in (\ref{SBPD2Form}) are positive semi-definite; and then
	\item the remaining free parameters can be zeroed.
  \end{itemize}
  
  The form (\ref{SBPD2Form}) leads to nonlinear equations and hence multiple families of solutions. Each one of these families can be optimized with the constraint that the $\mat{C}$ matrices are positive semi-definite. Some of these families are more difficult to optimize than others, in particular for hybrid Gauss-trapezoidal-Lobatto and hybrid Gauss-trapezoidal operators. Here we take the path of least resistance and choose one family for each operator that is easily optimized by Maple. 
\subsection{Diagonal-norm GSBP operators on pseudo-spectral nodal distributions}
We derive a number of diagonal-norm GSBP operators on nodal distributions associated with pseudo-spectral methods. Definition \ref{DefD2C} requires the construction of $\mat{R}\left(\mat{B}\right)$. In the most general case, $\mat{R}\left(\mat{B}\right)$ can be constructed as
\begin{equation}\label{linearR}
\mat{R}(\mat{B})=\sum_{i=1}^{N}\mat{B}(i,i)\mat{R}_{i},
\end{equation}
with the restriction that $\mat{R}_{i}$ is symmetric positive semi-definite. This formulation leads to linear degree equations (\ref{accD2V}), but nonlinear constraints from $\mat{R}_{i}$ to be symmetric positive semi-definite. Alternatively, $\mat{R}_{i}$ is constructed to be symmetric positive semi-definite as follows:
\begin{equation}\label{nonlinearR}
\mat{R}_{i}=\mat{L}^{T}_{i}\mat{\Lambda}_{i}\mat{L}_{i},
\end{equation}
where $\mat{L}_{i}$ is lower unitriangular and $\mat{\Lambda}_{i}$ is a diagonal matrix. Now the constraint that $\mat{R}_{i}$ be symmetric positive semi-definite reduces to the constraint that $\mat{\Lambda}_{i}$ be positive semi-definite; however, the degree equations become nonlinear. Although (\ref{nonlinearR}) is guaranteed to result in compatible order-matched operators, if solutions can be found, the resultant system of equations is very difficult to solve, particularly for operators with many nodes. This motivates the search for simplifications of $\mat{R}\left(\mat{B}\right)$, as have been found for classical FD-SBP operators; this is a current area of research.

We seek a construction of $\mat{R}\left(\mat{B}\right)$ such that it is of the form (\ref{linearR}) and satisfies the requirement that $\mat{R}_{i}$ be symmetric positive semi-definite, but avoids solving a large system of nonlinear equations. We can do this by taking advantage of Theorem \ref{CorollaryTheorem2}. We first solve for the constant-coefficient order-matched GSBP operator for the second derivative, given by
\begin{equation}\label{constantD2}
\mat{D}_{2}^{(p)}\left(\mat{B}\right)=\mat{H}^{-1}\left[
-\left(\mat{D}_{1}^{(p)}\right)^{T}\mat{H}\mat{D}_{1}^{(p)}-\mat{R}_{c}+\mat{E}\mat{D}_{b}^{(\ge p+1)}
\right],
\end{equation}
which has degree equations
   \begin{equation}\label{accuracyD2C}
    \mat{D}_{2}^{(q)}\mb{x}^{k}=k(k-1)\mb{x}^{k-2},\quad j\in[0,p].
    \end{equation}
\noindent By Theorem \ref{CorollaryTheorem2}, if $\mat{R}_{c}$ is symmetric positive semi-definite, then a compatible order-matched GSBP operator is given by
\begin{equation}\label{Form2}
\mat{D}_{2}^{(p)}\left(\mat{B}\right)=\mat{H}^{-1}\left[
-\left(\mat{D}_{1}^{(p)}\right)^{T}\mat{H}\mat{B}\mat{D}_{1}^{(p)}-\sum_{i=1}^{n}\frac{\mat{B}(i,i)}{n}\mat{R}_{c}+\mat{E}\mat{B}\mat{D}_{b}^{(\ge p+1)}
\right].
\end{equation}

The general steps to construct order-matched GSBP operators are as follows:
\begin{itemize}
 \item solve the degree equations (\ref{accuracy2}) for the first-derivative GSBP operator;
    \item if there are free parameters, optimize using (\ref{optimaD1D});
    \item if any free parameters remain, set them to zero;
    \item solve the degree equations for the constant-coefficient second derivative (\ref{accuracyD2C});
    \item use free parameters to ensure that $\mat{R}_{c}$ is positive semi-definite; 
    \item if form (\ref{Form2}) is used with an $\mat{R}_{c}$ that is not positive semi-definite, it is necessary to check if the $\mat{M}_{i}$ are positive semi-definite; and then
    \item if there are free parameters, the operator is optimized using (\ref{objD2}), and any remaining free parameters are set to zero.
\end{itemize}

As examples, we construct operators on the following pseudo-spectral nodal distributions:
 \begin{itemize}
 	\item Newton-Cotes, i.e.\ equally spaced nodal distribution
	\item Chebyshev-Gauss
		 \begin{equation}
 			x_{k} = -\cos\left(\frac{\left(2k+1\right)\pi}{2\left(N-1\right)+2}\right),\quad k\in[0,N-1]
		 \end{equation}
	\item Chebyshev-Lobatto
		\begin{equation}
			x_{k}=-\cos\left(\frac{k\pi}{N-1}\right),\quad k\in[0,N-1]
		\end{equation}
 \end{itemize}
 Even though the first-derivative GSBP operators are constructed on pseudo-spectral nodal distributions, the operators that are obtained are not the classical pseudo-spectral operators associated with those nodal distributions, which have dense norms \cite{Shen2011}.

 \subsection{Summary of operators}
 Table \ref{ABBGSBP} lists the abbreviations used to refer to the various GSBP operators used in Section \ref{Numerical Results}. The degree and order of the various operators, on pseudo-spectral nodal distributions with $n$ nodes, for the application of the first-derivative operator twice are given as
 \begin{equation}
 \begin{array}{lr}
 \textrm{degree}=\lceil\frac{n}{2}\rceil,\textrm{ and}&\textrm{order}=\lceil\frac{n}{2}\rceil-1,
 \end{array}
 \end{equation}
where $\lceil\cdot\rceil$ is the ceiling operator which returns the closest integer larger than the argument, while for order-matched operators, the relationship is given as
  \begin{equation}
 \begin{array}{lr}
 \textrm{ degree}=\lceil\frac{n}{2}\rceil+1,\textrm{ and}&\textrm{order}=\lceil\frac{n}{2}\rceil.
 \end{array}
 \end{equation}
 \noindent For operators with a repeating interior stencil, the following relations hold for the application of the first-derivative operator twice:
 
 \begin{equation}
 \begin{array}{lr}
 \textrm{degree}=p,\textrm{ and}&\textrm{order}=p-1,
 \end{array}
 \end{equation}
 while for order-matched operators, the relationship is given as
  \begin{equation}
 \begin{array}{lr}
 \textrm{ degree}=p+1,\textrm{ and}&\textrm{order}=p.
 \end{array}
 \end{equation}
 \begin{table}[t!]
\caption{Abbreviations for GSBP operators}
\begin{center}
\begin{tabular}{l|l}
Abbreviation&Operator\\
\hline
$NC[n]1$&\parbox[t]{9cm}{Application of the first-derivative operator twice on $n$ Newton-Cotes quadrature nodes}\\
$NC[n]2$&\parbox[t]{9cm}{Compatible and order-matched operator on $n$ Newton-Cotes quadrature nodes constructed using form (\ref{Form2})}\\
$CGL[n]1$&\parbox[t]{9cm}{Application of the first-derivative operator twice on $n$ Chebyshev-Gauss-Lobatto quadrature nodes}\\
$CGL[n]2$&\parbox[t]{9cm}{Non-compatible and order-matched operator on $n$ Chebyshev-Gauss-Lobatto quadrature nodes constructed using form (\ref{constantD2}) for constant coefficients}\\
$CG[n]1$&\parbox[t]{9cm}{Application of the first-derivative operator twice on $n$ Chebyshev-Gauss quadrature nodes}\\
$CG[n]2$&\parbox[t]{9cm}{Compatible and order-matched operator on $n$ Chebyshev-Gauss quadrature nodes constructed using form (\ref{Form2})}\\
$CSBP[p]1$&\parbox[t]{9cm}{Application of the first-derivative operator twice using classical SBP operators}\\
$CSBP[p]2$&\parbox[t]{9cm}{Compatible and order-matched classical SBP operator constructed using form (\ref{SBPD2Form})}\\
$HGTLSBP[p]1$&\parbox[t]{9cm}{Application of the first-derivative operator twice on Hybrid Gauss-trapezoidal-Lobatto quadrature nodes}\\
$HGTLSBP[p]2$&\parbox[t]{9cm}{Compatible and order-matched operator on hybrid Gauss-trapezoidal-Lobatto quadrature nodes constructed using form (\ref{SBPD2Form})}\\
$HGTSBP[p]1$&\parbox[t]{9cm}{Application of the first-derivative operator twice on hybrid Gauss-trapezoidal quadrature nodes}\\
$HGTSBP[p]2$&\parbox[t]{9cm}{Compatible and order-matched operator on hybrid Gauss-trapezoidal quadrature nodes constructed using form (\ref{SBPD2Form})}
\end{tabular}
\end{center}
\label{ABBGSBP}
\end{table}%

\section{Numerical results}\label{Numerical Results}
In this section, various GSBP operators are characterized in the context of the steady linear convection-diffusion equation given as
\begin{equation}\label{LCD}
-\frac{\partial \fnc{U}}{\partial x}+\frac{\partial}{\partial x}\left(\fnc{B}\frac{\partial \fnc{U}}{\partial x}\right)+\fnc{S}=0,\; x\in[x_{L},x_{R}],\;\fnc{B}>0,
\end{equation}
with boundary conditions
\begin{equation}\label{BCICLCD}
\begin{array}{l}
\alpha_{x_{L}}\fnc{U}_{x_{L}}+\beta_{x_{L}}\fnc{B}_{x_{L}}\left.\frac{\partial\fnc{U}}{\partial x}\right|_{x_{L}}=\fnc{G}_{x_{L}},\textrm{ and}\\\\
\alpha_{x_{R}}\fnc{U}_{x_{R}}+\beta_{x_{R}}\fnc{B}_{x_{R}}\left.\frac{\partial\fnc{U}}{\partial x}\right|_{x_{R}}=\fnc{G}_{x_{R}}.
\end{array}
\end{equation}
The variable coefficient is given by
\begin{equation}\label{SourceN}
\fnc{B}=3+\sin\left(x\right)\mr{e}^{\left(-\frac{x^{2}}{10}\right)},
\end{equation}
and $\fnc{B}=1$ for the constant-coefficient equation. The source term $\fnc{S}$ and the functions $\fnc{G}_{x_{L}}$ and $\fnc{G}_{x_{R}}$ are chosen such that the solution to (\ref{LCD}) and (\ref{BCICLCD}) is
\begin{equation}\label{SolutionN}
\begin{array}{rcl}
\fnc{U}(x)&=&\mr{e}^{\left(-\frac{x^{2}}{10}\right)}\cos \left( \pi \,x \right)  \left( 1-x\right).
\end{array}
\end{equation} 

The semi-discrete equations for a single block or element are given as
\begin{equation}\label{SDLCD}
\begin{array}{rcl}
\frac{\mr{d}\mb{u}_{h}}{\mr{d}t}&=&-\mat{D}_{1}^{(p)}\mb{u}_{h}+\mat{H}^{-1}\left\{-\left(\mat{D}_{1}^{(p)}\right)^{T}\mat{H}\mat{B}\mat{D}_{1}^{(p)}-\mat{R}(\mat{B})+\mat{E}\mat{B}\mat{D}_{b}^{(\ge p+1)}\right\}\mb{u}_{h}\\\\
&&+\mb{SAT}_{x_{L}}+\mb{SAT}_{x_{R}}+\mb{S},
\end{array}
\end{equation}
where the additional two terms are the SATs to impose the boundary conditions, and $\mb{S}$ is the projection of the source term onto the nodes. The SATs are constructed to mimic the continuous boundary conditions (\ref{BCICLCD}) and have the form
\begin{equation}\label{SATs}
\begin{array}{l}
\mb{SAT}_{x_{L}}=\sigma_{x_{L}}\mat{H}^{-1}\mat{E}_{x_{L}}\left(\alpha_{x_{L}}\mb{u}_{h}+\beta_{x_{L}}\mat{B}\mat{D}_{b}^{(\ge p+1)}\mb{u}_{h}-\mb{1}\fnc{G}_{x_{L}}\right),
\\\\
\mb{SAT}_{x_{R}}=\sigma_{x_{R}}\mat{H}^{-1}\mat{E}_{x_{R}}\left(\alpha_{x_{R}}\mb{u}_{h}+\beta_{x_{R}}\mat{B}\mat{D}_{b}^{(\ge p+1)}\mb{u}_{h}-\mb{1}\fnc{G}_{x_{R}}\right),
\end{array}
\end{equation}
where $\mb{1}$ is a vector of ones. The terms within the parentheses are an approximation of the boundary conditions; the rest of the SAT is constructed to allow the energy method to be applied, with the additional parameters $\sigma$ chosen so that an energy  estimate analogous to the continuous estimate can be constructed (for more information about the SATs used in this section see \cite{Gong2011}).

The extension to a multi-element approach necessitates SATs for inter-element coupling, in addition to the boundary SATs (\ref{SATs}). Consider two abutting elements, with solution $\mb{u}_{h}$ in the left element and solution $\mb{v}_{h}$ in the right element. The SAT for the right boundary of the left element is \cite{Gong2011}
\begin{equation}\label{SATu}
\begin{array}{rcl}
\mb{SAT}_{\mb{u}_{h}}&=&
\sigma_{1}^{(\mb{u}_{h})}\mat{H}_{\mb{u}_{h}}^{-1}\left(\mat{E}_{\mb{u}_{h},x_{R}}\mb{u}_{h}-\mb{t}_{\mb{u}_{h},x_{R}}\mb{t}_{\mb{v}_{h},x_{L}}^{T}\mb{v}_{h}\right)\\\\
&&+\sigma_{2}^{(\mb{u}_{h})}\mat{H}_{\mb{u}_{h}}^{-1}\left(\mat{E}_{\mb{u}_{h},x_{R}}\mat{B}_{\mb{u}_{h}}\mat{D}_{b,\mb{u}_{h}}^{(\ge p+1)}\mb{u}_{h}-\mb{t}_{\mb{u}_{h},x_{R}}\mb{t}_{\mb{v}_{h},x_{L}}^{T}\mat{B}_{\mb{v}_{h}}\mat{D}_{b,\mb{v}_{h}}^{(\ge p+1)}\mb{v}_{h}\right)
\\\\
&&+\sigma_{3}^{(\mb{u}_{h})}\mat{H}_{\mb{u}_{h}}^{-1}\left(\mat{D}_{b,\mb{u}_{h}}^{(\ge p+1)}\right)^{T}\mat{B}_{\mb{u}_{h}}\left(\mat{E}_{\mb{u}_{h},x_{R}}\mb{u}_{h}-\mb{t}_{\mb{u}_{h},x_{R}}\mb{t}_{\mb{v}_{h},x_{L}}^{T}\mb{v}_{h}\right),
\end{array}
\end{equation}
where the subscripts $\mb{u}_{h}$ and $\mb{v}_{h}$ are used to identify operators for the left and right elements, and
\begin{equation}
\begin{array}{rcl}
\mat{Q}_{\mb{u}_{h}}+\mat{Q}_{\mb{u}_{h}}^{T}&=&\mat{E}_{\mb{u}_{h},x_{R}}-\mat{E}_{\mb{u}_{h},x_{L}}=\mb{t}_{\mb{u}_{h},x_{R}}\mb{t}_{\mb{u}_{h},x_{R}}^{T}-\mb{t}_{\mb{u}_{h},x_{L}}\mb{t}_{\mb{u}_{h},x_{L}}^{T},\\\\
\mat{Q}_{\mb{v}_{h}}+\mat{Q}_{\mb{v}_{h}}^{T}&=&\mat{E}_{\mb{v}_{h},x_{R}}-\mat{E}_{\mb{v}_{h},x_{L}}=\mb{t}_{\mb{v}_{h},x_{R}}\mb{t}_{\mb{v}_{h},x_{R}}^{T}-\mb{t}_{\mb{v}_{h},x_{L}}\mb{t}_{\mb{v}_{h},x_{L}}^{T}.
\end{array}
\end{equation}
The SAT for the right element is given as
\begin{equation}\label{SATv}
\begin{array}{rcl}
\mb{SAT}_{\mb{v}_{h}}&=&
\sigma_{1}^{(\mb{v}_{h})}\mat{H}_{\mb{v}_{h}}^{-1}\left(\mat{E}_{\mb{v}_{h},x_{L}}\mb{v}_{h}-\mb{t}_{\mb{v}_{h},x_{L}}\mb{t}_{\mb{u}_{h},x_{R}}^{T}\mb{u}_{h}\right)\\\\
&&+\sigma_{2}^{(\mb{v}_{h})}\mat{H}_{\mb{v}_{h}}^{-1}\left(\mat{E}_{\mb{v}_{h},x_{L}}\mat{B}_{\mb{v}_{h}}\mat{D}_{b,\mb{v}_{h}}^{(\ge p+1)}\mb{v}_{h}-\mb{t}_{\mb{v}_{h},x_{L}}\mb{t}_{\mb{u}_{h},x_{R}}^{T}\mat{B}_{\mb{u}_{h}}\mat{D}_{b,\mb{u}_{h}}^{(\ge p+1)}\mb{u}_{h}\right)
\\\\
&&+\sigma_{3}^{(\mb{v}_{h})}\mat{H}_{\mb{v}_{h}}^{-1}\left(\mat{D}_{b,\mb{v}_{h}}^{(\ge p+1)}\right)^{T}\mat{B}_{\mb{v}_{h}}\left(\mat{E}_{\mb{v}_{h},x_{L}}\mb{v}_{h}-\mb{t}_{\mb{v}_{h},x_{L}}\mb{t}_{\mb{u}_{h},x_{R}}^{T}\mb{u}_{h}\right).
\end{array}
\end{equation}

For the numerical studies in this paper, the following values for the boundary parameters and the SAT parameters, based on those in \cite{Gong2011}, are used:
\begin{equation}
\begin{array}{cccc}
\alpha_{x_{L}}=-1&\beta_{x_{L}}=1&\alpha_{x_{R}}=0&\beta_{x_{L}}=-1\\\\
\tau_{x_{L}}=1&\tau_{x_{R}}=1\\\\
\sigma_{1}^{(\mb{u})_{h}}=\frac{1}{2}&\sigma_{2}^{(\mb{u})_{h}}=1&\sigma_{3}^{(\mb{u})_{h}}=-2\\\\
\sigma_{1}^{(\mb{v})_{h}}=-\frac{1}{2}&\sigma_{2}^{(\mb{v})_{h}}=2&\sigma_{3}^{(\mb{v})_{h}}=-2.
\end{array}
\end{equation}
The solution error is defined by
\begin{equation}\label{solnerror}
\|\mb{e}\|_{\mat{H}} = \sqrt{\mb{e}^{T}\bar{\mat{H}}\mb{e}},
\end{equation}
where $\mb{e}=\left(\mb{u}_{h}-\mb{u}_{a}\right)$, $\mb{u}_{a}$ is the restriction of the analytical solution onto the grid, and $\bar{\mat{H}}$ is a block diagonal matrix with the norm matrix, $\mat{H}$, from each element along the diagonal.
\subsection{Constant-coefficient linear convection-diffusion equation}
Figures \ref{CGL_CONV_CONST} to \ref{CSBP_CONV_CONST} display the convergence of the various operators for the constant-coefficient problem, where operators with a repeating interior stencil have been implemented using a traditional approach. Convergence rates of the SBP norm of the error, calculated by (\ref{solnerror}), are given in Table \ref{tableConstant}, where the rates have been computed as the slope of the best line of fit through the highlighted points in the figures.

All of the order-matched operators have smaller global error than operators constructed as the application of the first-derivative operator twice. For operators on pseudo-spectral nodal distribution, most have rates of convergence equal to the order plus $2$, in line with what is expected for classical FD-SBP operators for solving parabolic problems \cite{Svard2006}. The exceptions are $CGL52$, $CG52$, $NC42$, and $NC61$, which have convergence rates between $0.5$ to $1$ less than expected.

For operators with a repeating interior stencil, the convergence rates are equal to or better than predicted, with the exception of $HGT31$, which is half an order worse than expected. For $p=2$, there does not appear to be any benefit for the derived order-matched hybrid Gauss-trapezoidal-Lobatto and hybrid Gauss-trapezoidal operators. On the other hand, for operators constructed as the application of the first-derivative operator twice the hybrid Gauss-trapezoidal-Lobatto and hybrid Gauss-trapezoidal operators have smaller global error compared to the classical FD-SBP operator, with the hybrid Gauss-trapezoidal operator having the smallest global error. For $p=3$ both the order-matched and the application of the first-derivative operator twice for hybrid Gauss-trapezoidal-Lobatto and hybrid Gauss-trapezoidal operators are better than the classical FD-SBP operators, with hybrid Gauss-trapezoidal having the smallest global error. Finally, for $p=4$, the hybrid Gauss-trapezoidal-Lobatto operator seems to show round-off error very early on (this was even worse for the hybrid Gauss-trapezoidal operator, for $p=4$, not shown here)---examining the coefficients of these operators, they are very large, and as a result, such operators are likely prone to round-off error. It appears that for these operators, in addition to the criteria used in this paper for optimization, it will be necessary to carefully monitor the size of the coefficients of the operators in order to reduce round-off error. On the other hand, the classical operators for $p=4$ face no such issues, with the compatible order-matched operator performing significantly better than the application of the first-derivative operator twice.

The numerical simulations show that order-matched GSBP as well as classical minimum-stencil FD-SBP operators have preferential error properties, compared to the application of the first-derivative operator twice. In particular, the results show that they lead to both a reduction in the global error as well as an increase in the convergence rate. Moreover, the hybrid Gauss-trapezoidal-Lobatto and hybrid Gauss-trapezoidal operators have competitive error characteristics compared to classical FD-SBP operators and represent an attractive means of taking advantage of the GSBP concept within the context of existing FD-SBP codes.  

\subsection{Variable coefficient linear convection-diffusion equation}
Figures \ref{CGL_CONV_VAR} and \ref{CSBP_CONV_VAR} display the convergence of the various operators for the variable-coefficient problem, where all operators were implemented using an element approach. Convergence rates of the SBP norm of the error, calculated by (\ref{solnerror}), are given in Table \ref{tableVar}, where the rates have been computed as the slope of the best line of fit through the highlighted points in the figures.

The general pattern is similar to the constant-coefficient problem; with few exceptions, the order-matched operators have smaller global error compared to the application of the first-derivative operator twice. The numerical simulations show that the proposed construction of compatible order-matched GSBP operators on pseudo-spectral nodal distributions using (\ref{Form2}) retains the preferential error characteristics observed for order-matched operators for the constant-coefficient linear convection-diffusion equation. For element-based operators, this implies that, for parabolic problems, order-matched operators are more efficient that the application of the first-derivative operator twice, since both have the same number of floating-point operations. For operators with a repeating interior stencil, the situation is less clear; on the one hand, for all $p$, the order-matched operators have preferential error characteristics, on the other hand, for $p< 3$, they require more floating-point operations to compute the derivatives on the interior.
\begin{figure}[t!]
\begin{center}
\subfigure[]{\includegraphics[width = 12cm]{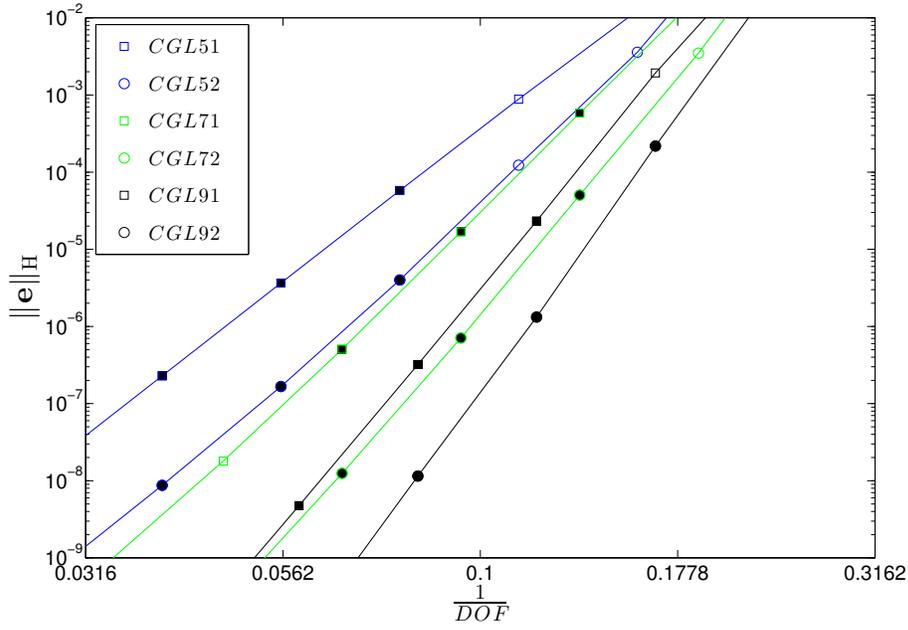}}
\subfigure[]{\includegraphics[width = 12cm]{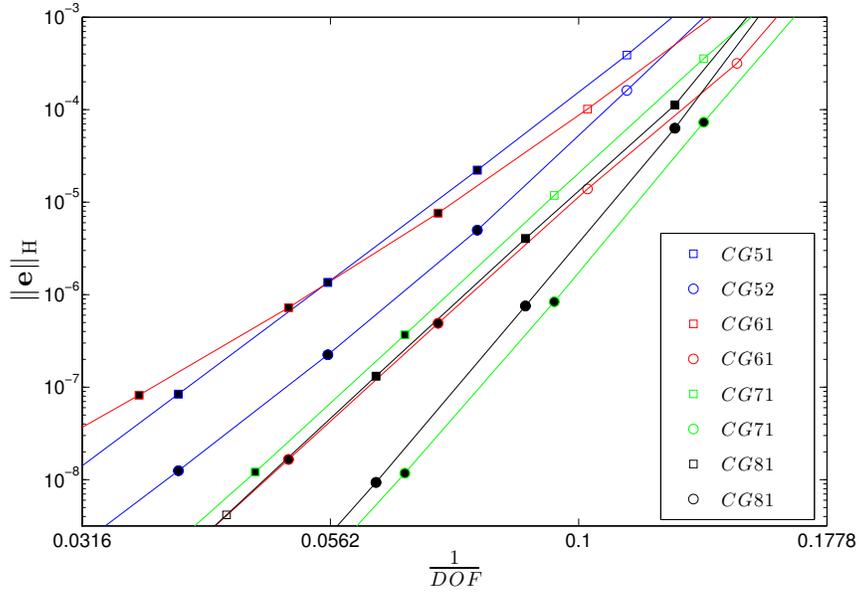}}
\end{center}
\caption{$\lVert \mathbf{e}\rVert_{\mat{H}}$ of the solution to (\ref{SDLCD}) and (\ref{BCICLCD}) with constant coefficients using a) Chebyshev-Gauss-Lobatto quadrature nodes, and b) Chebyshev-Gauss quadrature nodes. Filled in markers represent the points used to compute the rate of convergence.}
\label{CGL_CONV_CONST}
\end{figure}

\begin{figure}[t!]
\begin{center}
\includegraphics[width = 12cm]{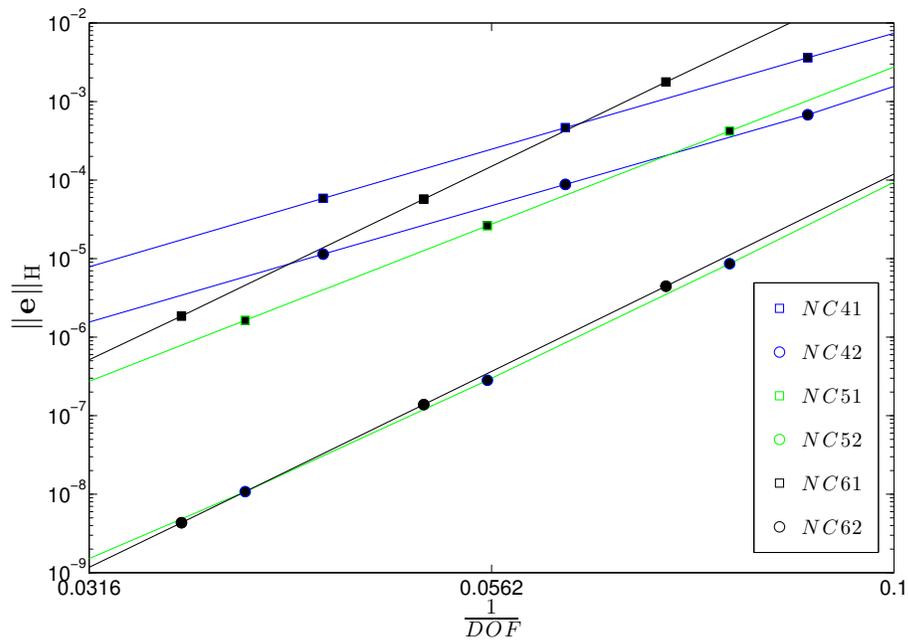}\\
\end{center}
\caption{$\lVert \mathbf{e}\rVert_{\mat{H}}$ of the solution to (\ref{SDLCD}) and (\ref{BCICLCD}) with constant coefficients using Newton-Cotes quadrature nodes. Filled in markers represent the points used to compute the rate of convergence.}
\label{CG_CONV_CONST}
\end{figure}

\begin{figure}[t!]
\begin{center}
\subfigure[]{\includegraphics[width = 9cm]{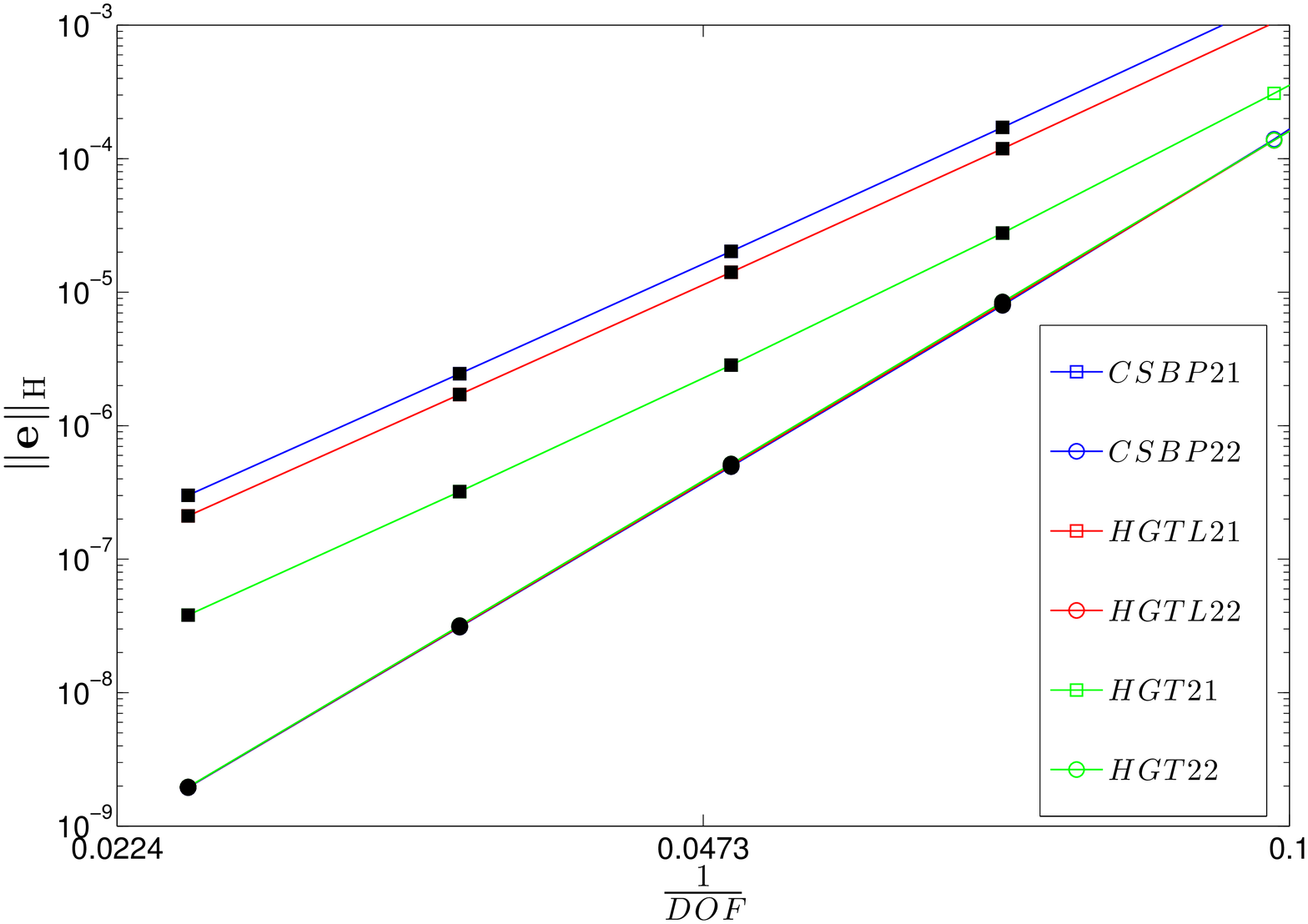}}
\subfigure[]{\includegraphics[width = 9cm]{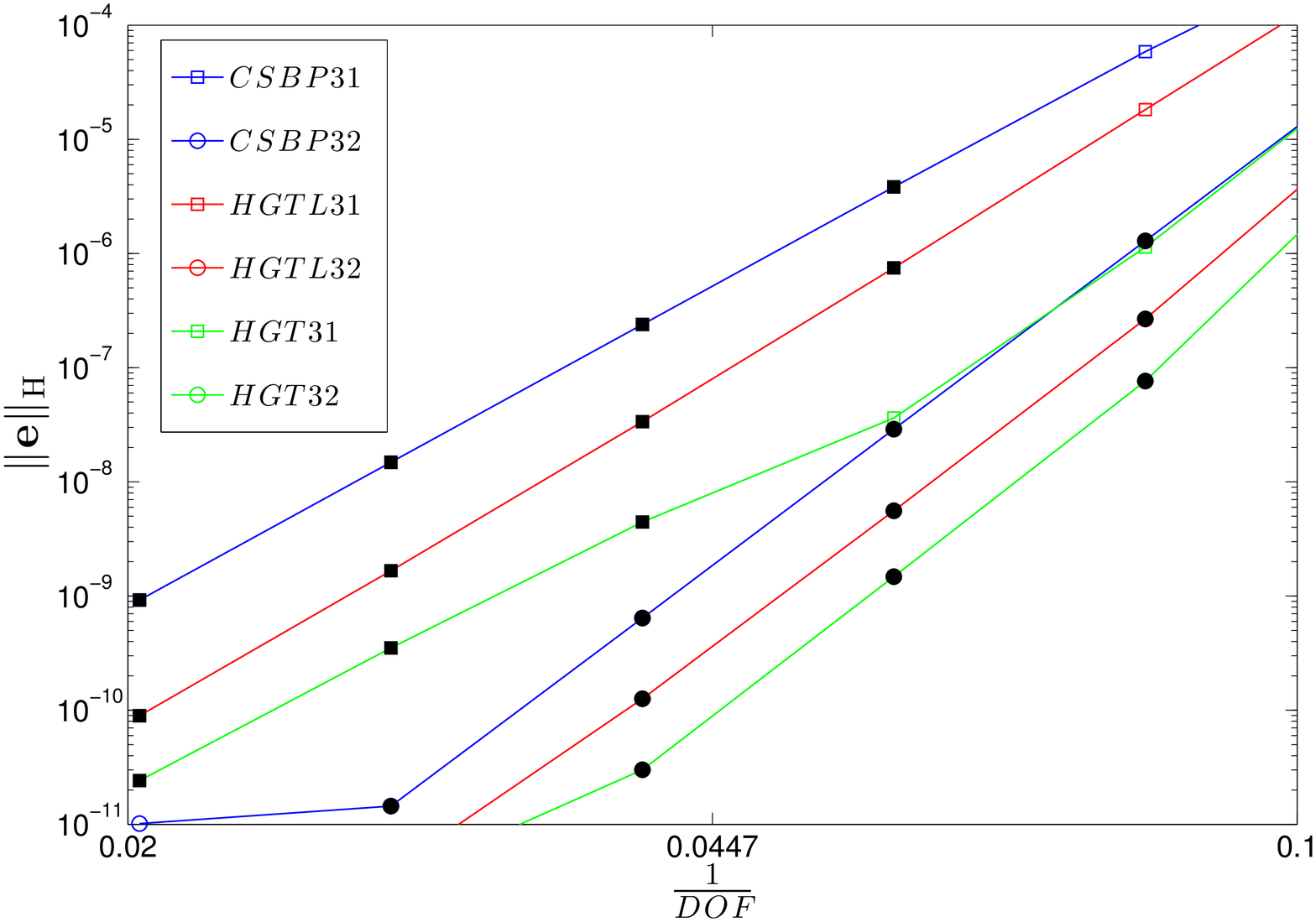}}\\
\subfigure[]{\includegraphics[width = 9cm]{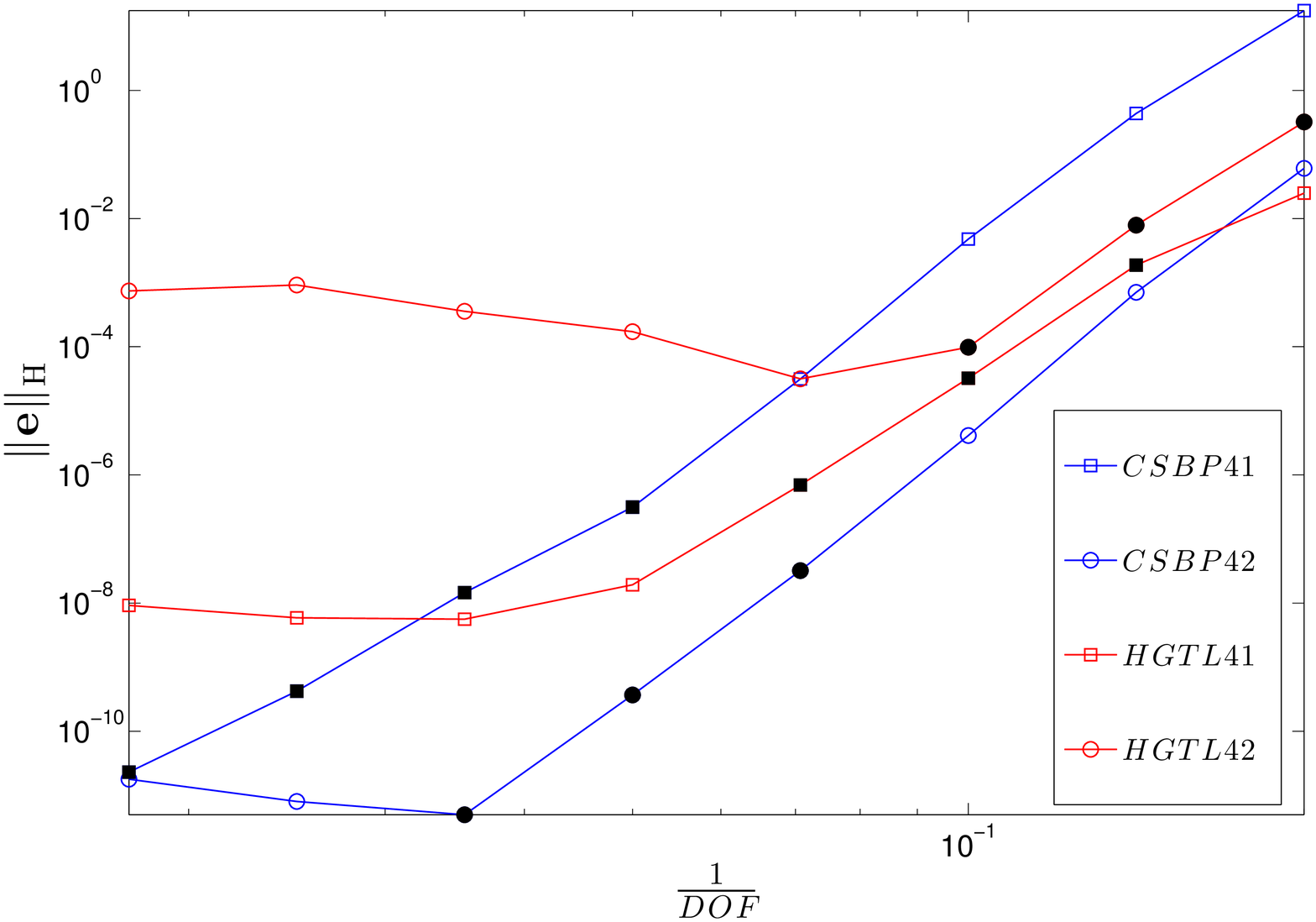}}
\end{center}
\caption{$\lVert \mathbf{e}\rVert_{\mat{H}}$ of the solution to (\ref{SDLCD}) and (\ref{BCICLCD}) with constant coefficients using operators with a repeating interior stencil with a) $p=2$, b) $p=3$, and c) $p=4$. Filled in markers represent the points used to compute the rate of convergence.}
\label{CSBP_CONV_CONST}
\end{figure}
\begin{landscape}
\begin{table}[htdp]
\caption{Order of truncation term for the solution of (\ref{SLCD}) and (\ref{BCICLCD}) with a constant coefficient}
\begin{center}
\begin{tabular}{cccccccccccc}
\hline\hline
Operator&Order&Operator&Order&Operator&Order&Operator&Order&Operator&Order&Operator&Order\\
\hline
CGL51&3.99&CG51&4.0208&NC41&2.9744&CSBP21&3.0517&CSBP31&4.0072&CSBP41&4.633\\
CGL52&4.4206&CG52&4.3191&NC42&2.9479&CSBP22&4&CSBP32&5.4836&CSBP42&6.3271\\
CGL71&5.0832&CG61&3.2687&NC51&4.0036&HGTL21&3.046&HGTL31&4.3422&HGTL41&5.6999\\
CGL72&5.9918&CG62&4.8767&NC52&4.8215&HGTL22&4.007&HGTL32&5.5252&HGTL42&5.839\\
CGL91&6.1273&CG71&4.8434&NC61&4.9503&HGT21&3.1661&HGT31&3.762\\
CGL92&7.1048&CG72&6.3012&NC62&5.0016&HGT22&4.0213&HGT32&5.6568\\
&&CG81&4.8698\\
&&CG82&6.3601\\
\hline
\end{tabular}
\end{center}
\label{tableConstant}
\end{table}

\begin{table}[htdp]
\caption{Order of truncation term for the solution of (\ref{SLCD}) and (\ref{BCICLCD}) with a variable coefficient}
\begin{center}
\begin{tabular}{cccccccccc}
\hline\hline
Operator&Order&Operator&Order&Operator&Order&Operator&Order&Operator&Order\\
\hline
CG51&4.0218&NC41&2.9841&CSBP21&3.055&CSBP31&3.993&CSBP41&5.6502\\
CG52&4.3093&NC42&3.1205&CSBP22&3.8867&CSBP32&5.3169&CSBP42&6.6792\\
CG61&3.2477&NC51&4.0042&HGTL21&3.0512&HGTL31&4.3196&HGTL41&5.8675\\
CG62&4.8252&NC52&4.8013&HGTL22&3.9834&HGTL32&5.2693&HGTL42&5.7629\\
CG71&4.8306&NC61&4.953&HGT21&3.1764&HGT31&3.6796\\
CG72&6.1742&NC62&4.979&HGT22&4.0926&HGT32&6.1433\\
CG81&4.8804\\
CG82&6.4079\\
\hline
\end{tabular}
\end{center}
\label{tableVar}
\end{table}
\end{landscape}
\begin{figure}[t!]
\begin{center}
\subfigure[]{\includegraphics[width = 12cm]{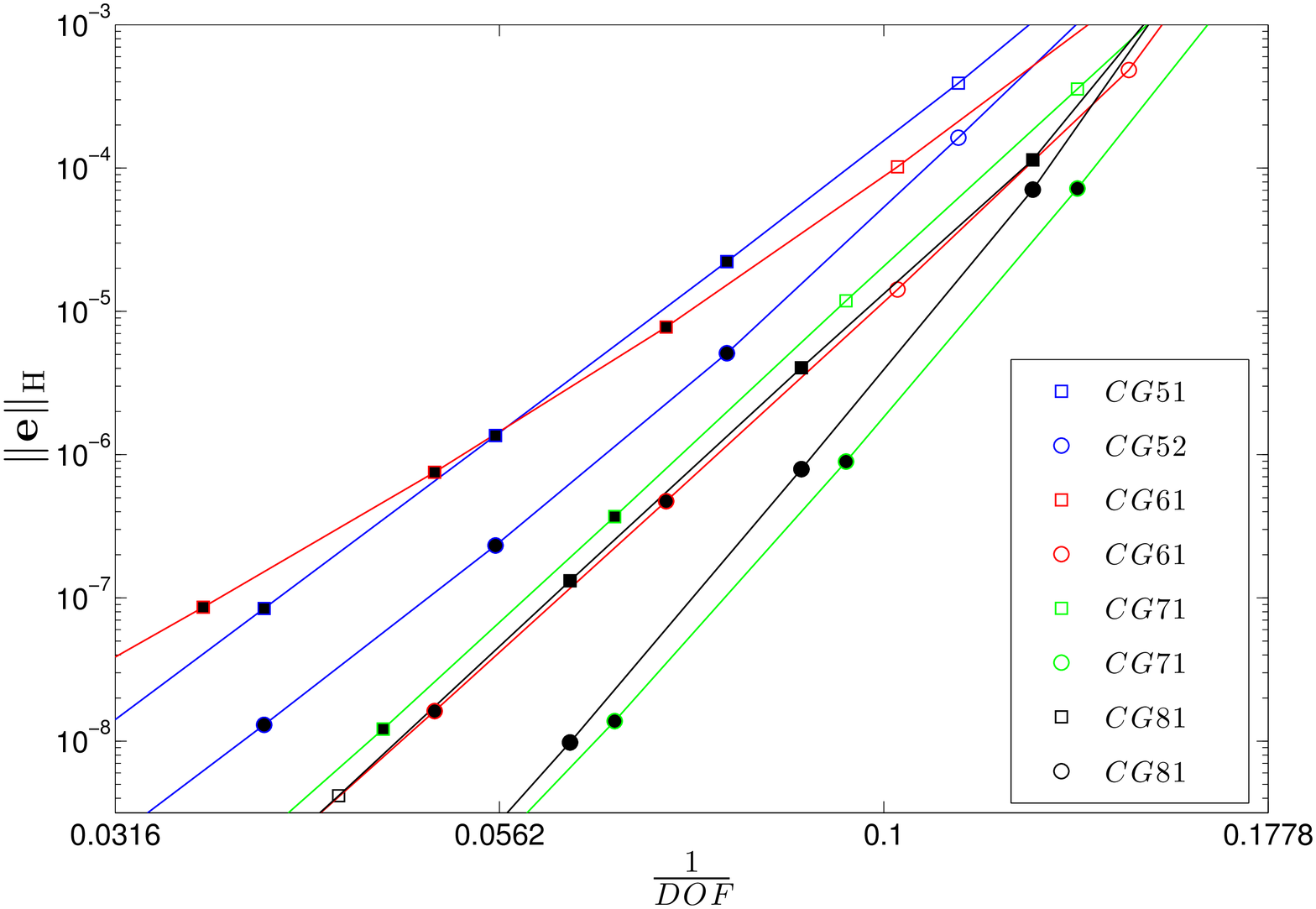}}
\subfigure[]{\includegraphics[width = 12cm]{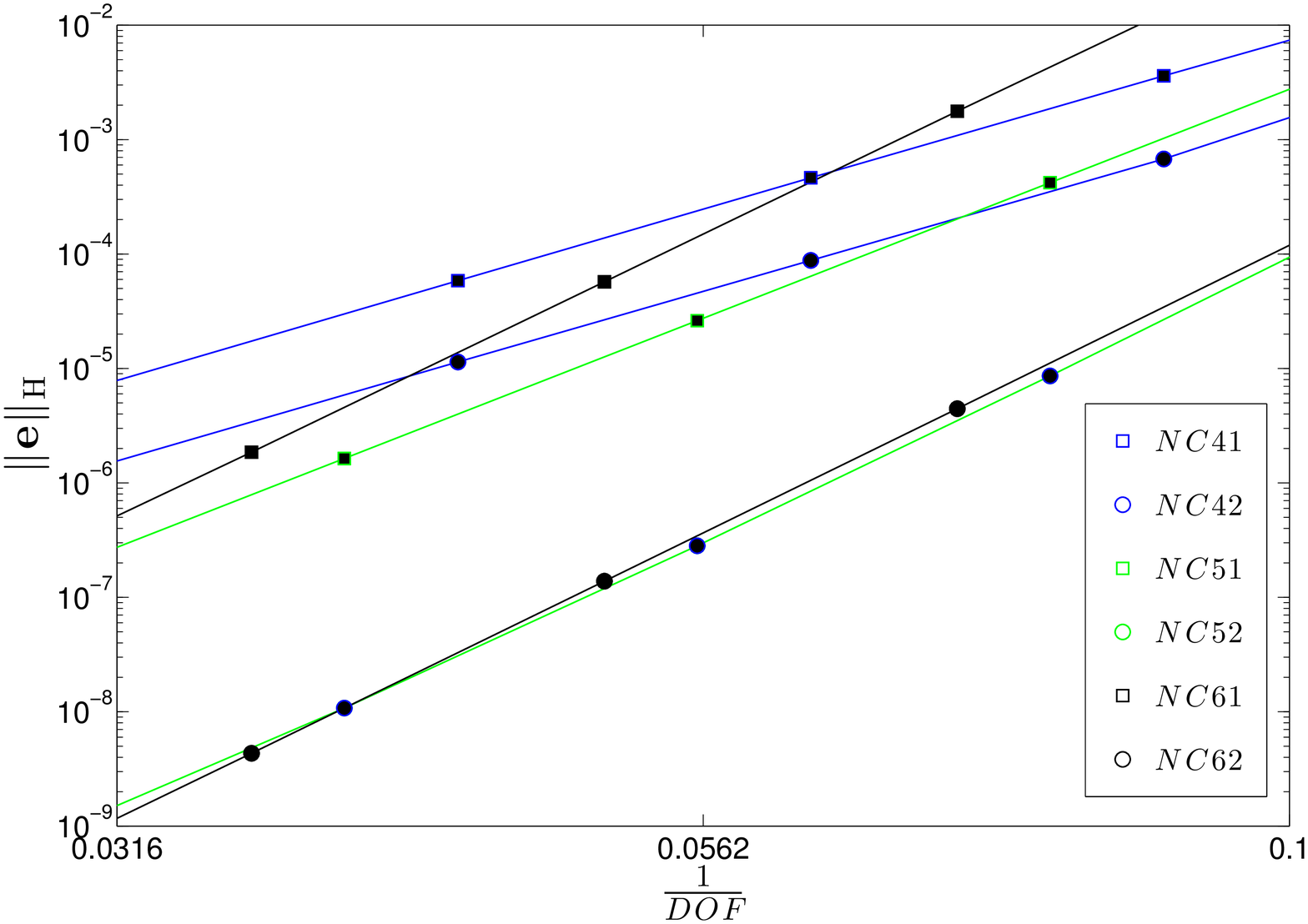}}
\end{center}
\caption{$\lVert \mathbf{e}\rVert_{\mat{H}}$ of the solution to (\ref{SDLCD}) and (\ref{BCICLCD}) with variable coefficients using a) Chebyshev-Gauss quadrature nodes, and b) Newton-Cotes quadrature nodes. Filled in markers represent the points used to compute the rate of convergence.}
\label{CGL_CONV_VAR}
\end{figure}

\begin{figure}[t!]
\begin{center}
\subfigure[]{\includegraphics[width = 9cm]{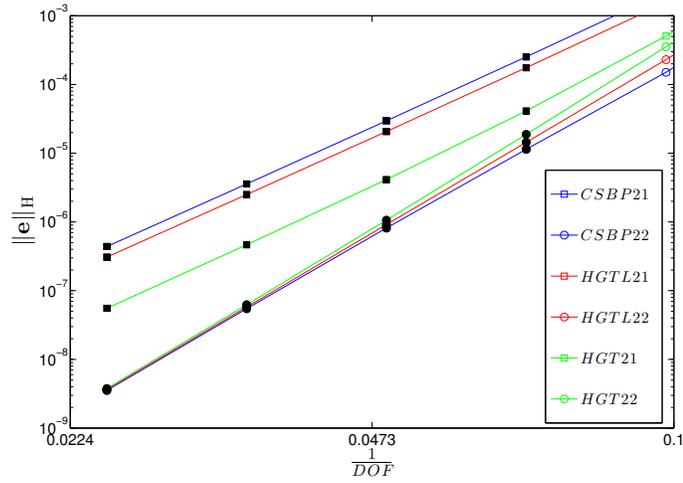}}
\subfigure[]{\includegraphics[width = 9cm]{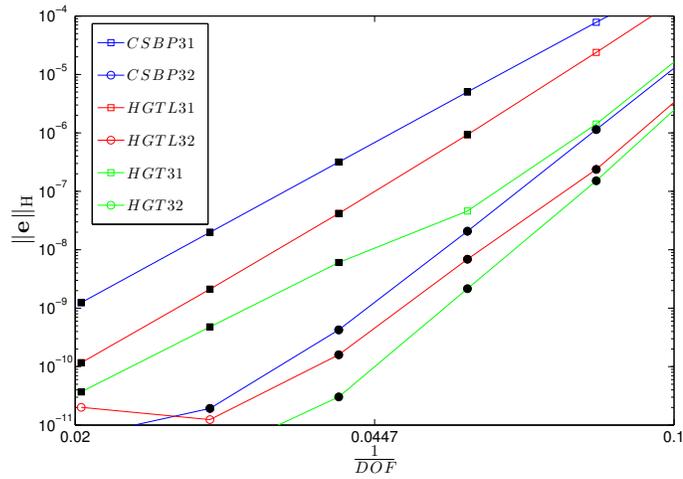}}\\
\subfigure[]{\includegraphics[width = 9cm]{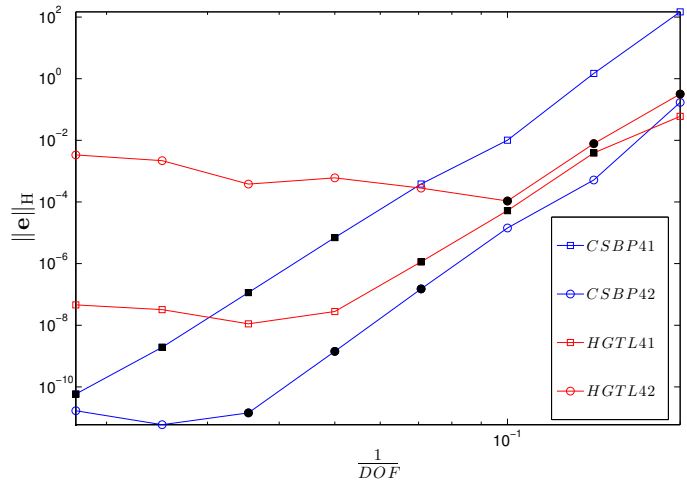}}
\end{center}
\caption{$\lVert \mathbf{e}\rVert_{\mat{H}}$ of the solution to (\ref{SDLCD}) and (\ref{BCICLCD}) with variable coefficients using operators with a repeating interior stencil with a) $p=2$, b) $p=3$, and c) $p=4$. Filled in markers represent the points used to compute the rate of convergence.}
\label{CSBP_CONV_VAR}
\end{figure}
\clearpage
\section{Conclusions and future work}\label{Conclusions and future work}
In this paper, we extended the generalized summation-by-parts operators developed in \cite{DCDRF2014} to the second derivative with variable coefficients. We proposed definitions for operators that are one order more accurate than the application of the first-derivative operator twice. Both non-compatible and compatible operators were considered. These operators were found to have preferential error properties, relative to the application of the first-derivative operator twice, for solving the linear convection-diffusion equation with constant and variable-coefficient second-derivative terms. The results suggest that for element-based operators, using order-matched operators should lead to more efficient discretization schemes. For operators with a repeating interior stencil, for $p>3$, this should also be the case. However, for $p<3$, though the operators are more accurate, they are more expensive to compute and further analysis is required. The following areas of future work follow naturally from this paper:
\begin{itemize}
	\item Extending GSBP operators for the second derivative to multi-dimensional operators that do not depend on Kronecker products; and
	\item Deriving SATs that lead to dual-consistent schemes for order-matched GSBP operators which could potentially lead to superconvergent functionals (see Hicken and Zingg \cite{Hicken2011b}).
\end{itemize}
\bibliographystyle{siam}
\bibliography{Bib}

\begin{thebibliography}{10}

\bibitem{Abarbanel2000}
{\sc Saul~S. Abarbanel, Alina~E. Chertock, and Amir Yefet}, {\em Strict
  stability of high-order compact implicit finite-difference schemes: The role
  of boundary conditions for hyperbolic pdes, i}, Journal of Computational
  Physics, 160 (2000), pp.~42--66.

\bibitem{Abarbanel2000b}
\leavevmode\vrule height 2pt depth -1.6pt width 23pt, {\em Strict stability of
  high-order compact implicit finite-difference schemes: the role of boundary
  conditions for hyperbolic {PDE}s ii}, Journal of Computational Physics, 160
  (2000), pp.~67--86.

\bibitem{Albin2014}
{\sc Nathan Albin and Joshua Klarmann}, {\em Existence of {SBP} operators with
  diagonal norm}, arXiv:1403.5750v1,  (2014).

\bibitem{Alpert1999}
{\sc Bradley~K. Alpert}, {\em Hybrid {G}auss-trapezoidal quadrature rules},
  SIAM Journal on Scientific Computing, 5 (1999), pp.~1551--1584.

\bibitem{Carpenter1996}
{\sc Mark~H. Carpenter and David Gottlieb}, {\em Spectral methods on arbitrary
  grids}, Journal of Computational Physics, 129 (1996), pp.~74--86.

\bibitem{Carpenter1994}
{\sc Mark~H. Carpenter, David Gottlieb, and Saul Abarbanel}, {\em Time-stable
  boundary conditions for finite-difference schemes solving hyperbolic systems:
  Methodology and application to high-order compact schemes}, Journal of
  Computational Physics, 111 (1994), pp.~220--236.

\bibitem{Carpenter1999}
{\sc Mark~H. Carpenter, Jan Nordstr\"om, and David Gottlieb}, {\em A stable and
  conservative interface treatment of arbitrary spatial accuracy}, Journal of
  Computational Physics, 148 (1999), pp.~341--365.

\bibitem{Chertock1998}
{\sc Alina~E. Chertock}, {\em Strict stability of high-order compact implicit
  finite-difference schemes: the role of boundary conditions for hyperbolic
  PDEs}, PhD thesis, Tel-Aviv Univsersity, 1998.

\bibitem{DCDRF2014}
{\sc David~C. {Del Rey Fern\'andez}, Pieter~D. Boom, and David~W. Zingg}, {\em
  A generalized framework for nodal first derivative summation-by-parts
  operators}, Journal of Computational Physics, 266 (2014), pp.~214--239.

\bibitem{Fernandez2014}
{\sc David~C. {Del Rey Fern\'andez}, Jason~E. Hicken, and David~W. Zingg}, {\em
  Review of summation-by-parts operators with simultaneous approximation terms
  for the numerical solution of partial differential equations}, Computers \&
  Fluids, 95 (2014), pp.~171--196.

\bibitem{Fernandez2012}
{\sc David~C. {Del Rey Fern\'andez} and David~W. Zingg}, {\em High-order
  compact-stencil summation-by-parts operators for the second derivative with
  variable coefficients}, ICCFD7-2803,  (2012).

\bibitem{Fernandez2013}
\leavevmode\vrule height 2pt depth -1.6pt width 23pt, {\em High-order
  compact-stencil summation-by-parts operators for the compressible
  {N}avier-{S}tokes equations}, AIAA Paper 2013-2570,  (2013), p.~{}.

\bibitem{Gassner2013}
{\sc Gregor~J. Gassner}, {\em A skew-symmetric discontinuous {G}alerkin
  spectral element discretization and its relation to {SBP}-{SAT} finite
  difference methods}, SIAM Journal on Scientific Computing, 35 (2013),
  pp.~A1233--A1253.

\bibitem{Gong2011}
{\sc Jing Gong and Jan Nordstr\"om}, {\em Interface procedures for finite
  difference approximations of the advection-diffusion equation}, Journal of
  Computational and Applied Mathematics, 236 (2011), pp.~602--620.

\bibitem{Gustafsson2008}
{\sc Bertil Gustafsson}, {\em High Order Difference Methods for Time Dependent
  {PDE}}, Springer, 2008.

\bibitem{Gustafsson2013}
{\sc Bertil Gustafsson, Heinz-Otto Kreiss, and Joseph Oliger}, {\em
  Time-Dependent Problems and Difference Methods}, Pure and Applied
  Mathematics, Wiley, second~ed., 2013.

\bibitem{Hesthaven1996}
{\sc Jan~S. Hesthaven and David Gottlieb}, {\em A stable penalty method for the
  compressible {N}avier-{S}tokes equations: {I}. open boundary conditions},
  SIAM Journal on Scientific Computing, 17 (1996), pp.~579--612.

\bibitem{Hicken2011b}
{\sc Jason~E. Hicken and David~W. Zingg}, {\em Superconvergent functional
  estimates from summation-by-parts finite-difference discretizations}, SIAM
  Journal on Scientific Computing, 33 (2011), pp.~893--922.

\bibitem{Kamakoti2009}
{\sc Ramji Kamakoti and Carlos Pantano}, {\em High-order narrow stencil
  finite-difference approximations of second-derivatives involving variable
  coefficients}, SIAM Journal on Scientific Computing, 31 (2009),
  pp.~4222--4243.

\bibitem{Kreiss2004}
{\sc Heinz-Otto Kreiss and Jens Lorenz}, {\em Initial-Boundary Value Problems
  and the {N}avier-{S}tokes Equations}, vol.~47 of Classics in Applied
  Mathematics, SIAM, 2004.

\bibitem{Kreiss1974}
{\sc Heinz-Otto Kreiss and Godela Scherer}, {\em Finite element and finite
  difference methods for hyperbolic partial differential equations}, in
  Mathematical aspects of finite elements in partial differential equations,
  Academic Press, New York/London, 1974, pp.~195--212.

\bibitem{Mattsson2012}
{\sc Ken Mattsson}, {\em Summation by parts operators for finite difference
  approximations of second-derivatives with variable coefficients}, Journal of
  Scientific Computing, 51 (2012), pp.~650--682.

\bibitem{Mattsson2013}
{\sc Ken Mattsson and Martin Almquist}, {\em A solution to the stability issues
  with block norm summation by parts operators}, Journal of Computational
  Physics, 15 (2013), pp.~418--442.

\bibitem{Mattsson2014}
{\sc Ken Mattsson, Martin Almquist, and Mark~H. Carpenter}, {\em Optimal
  diagonal-norm {SBP} operators}, Journal of Computational Physics, 264 (2014),
  pp.~91--111.

\bibitem{Mattsson2004b}
{\sc Ken Mattsson and Jan Nordstr\"om}, {\em Summation by parts operators for
  finite difference approximations of second derivatives}, Journal of
  Computational Physics, 199 (2004), pp.~503--540.

\bibitem{Mattsson2008}
{\sc Ken Mattsson, Magnus Sv\"ard, and Mohammad Shoeybi}, {\em Stable and
  accurate schemes for the compressible {N}avier-{S}tokes equations}, Journal
  of Computational Physics, 227 (2008), pp.~2293--2316.

\bibitem{Nordstrom1999}
{\sc Jan Nordstr\"om and Mark~H. Carpenter}, {\em Boundary and interface
  conditions for high-order finite-difference methods applied to the {E}uler
  and {N}avier-{S}tokes equations}, Journal of Computational Physics, 148
  (1999), pp.~621--645.

\bibitem{Nordstrom2001b}
\leavevmode\vrule height 2pt depth -1.6pt width 23pt, {\em High-order
  finite-difference methods, multidimensional linear problems, and curvilinear
  coordinates}, Journal of Computational Physics, 173 (2001), pp.~149--174.

\bibitem{Nordstrom2003}
{\sc Jan Nordstr\"om, Karl Forsberg, Carl Adamsson, and Peter Eliasson}, {\em
  Finite volume methods, unstructured meshes and strict stability for
  hyperbolic problems}, Applied Numerical Mathematics, 45 (2003), pp.~453--473.

\bibitem{Reichert2012}
{\sc Adam Reichert, Michel~T. Heath, and Daniel~J. Bodony}, {\em Energy stable
  numerical method for hyperbolic partial differential equations using
  overlapping domain decomposition}, Journal of Computational Physics, 231
  (2012), pp.~5243--5265.

\bibitem{Reichert2011}
{\sc Adam~Harold Reichert}, {\em Stable numerical methods for hyperbolic
  partial differential equations using overlapping domain decomposition}, PhD
  thesis, University of Illinois at Urbane-Champaign, 2011.

\bibitem{Shen2011}
{\sc Jien Shen, Tao Tang, and Li-Lian Wang}, {\em Spectral methods algorithms,
  analysis and applications}, springer, 2011.

\bibitem{Strand1994}
{\sc Bo~Strand}, {\em Summation by parts for finite difference approximations
  for d/dx}, Journal of Computational Physics, 110 (1994), pp.~47--67.

\bibitem{Svard2006}
{\sc Magnus Sv\"ard and Jan Nordstr\"om}, {\em On the order of accuracy for
  difference approximation of initial-boundary value problems}, Journal of
  Computational Physics, 218 (2006), pp.~333--352.

\bibitem{Svard2014}
\leavevmode\vrule height 2pt depth -1.6pt width 23pt, {\em Review of
  summation-by-parts schemes for initial-boundary-value-problems}, Journal of
  Computational Physics, 268 (2014), pp.~17--38.

\end{thebibliography}

\end{document}